\documentclass[11pt]{amsart}
\usepackage{amsmath,mathtools}
\usepackage{amsmath, amsthm, amssymb, amsfonts, enumerate, cancel}

\usepackage[colorlinks=true,linkcolor=blue,urlcolor=blue]{hyperref}
\usepackage{dsfont}
\usepackage{color}
\usepackage{geometry}
\usepackage{todonotes}
\usepackage{epstopdf}
\newcommand{\stkout}[1]{\ifmmode\text{\sout{\ensuremath{#1}}}\else\sout{#1}\fi}

\newtheorem{theorem}{Theorem}[section]
\newtheorem{remark}[theorem]{Remark}
\newtheorem{assumption}[theorem]{Assumption}
\newtheorem{lemma}[theorem]{Lemma}
\newtheorem{proposition}[theorem]{Proposition}
\newtheorem{corollary}[theorem]{Corollary}
\newtheorem{definition}[theorem]{Definition}

\def\A{\mathcal{A}}

\def \E{\mathsf{E}}

\def \P{\mathsf{P}}
\def \R{\mathbb{R}}
\def \F{\mathbb{F}}

\def\d{\mathrm{d}}

\newcommand{\Span}{\text{Span}}

\definecolor{red}{rgb}{1.0,0.0,0.0}

\definecolor{blu}{rgb}{0.0,0.0,1.0}

\definecolor{gre}{rgb}{0.03,0.50,0.03}

\title[Variational inequalities and smooth-fit principle in Hilbert spaces]{Variational inequalities and smooth-fit principle for singular stochastic control problems in Hilbert spaces} 
\author[Federico]{Salvatore Federico}
\author[Ferrari]{Giorgio Ferrari}
\author[Riedel]{Frank Riedel}
\author[R\"ockner]{Michael R\"ockner}

\address{S.~Federico: Dipartimento di Matematica, Universit\`a di Bologna,  Piazza di Porta S.\ Donato 5, 40126, Bologna, Italy}
\email{\href{mailto:s.federico@unibo.it}{s..federico@unibo.it}}
\address{G.~Ferrari: Center for Mathematical Economics (IMW), Bielefeld University, Universit\"atsstrasse 25, 33615, Bielefeld, Germany}
\email{\href{mailto:giorgio.ferrari@uni-bielefeld.de}{giorgio.ferrari@uni-bielefeld.de}}
\address{F.~Riedel: Center for Mathematical Economics (IMW), Bielefeld University, Universit\"atsstrasse 25, 33615, Bielefeld, Germany}
\email{\href{mailto:frank.riedel@uni-bielefeld.de}{frank.riedel@uni-bielefeld.de}}
\address{M.~R\"ockner: Faculty of Mathematics, Bielefeld University, Universit\"atsstrasse 25, 33615, Bielefeld, Germany}
\email{\href{mailto:roeckner@math.uni-bielefeld.de}{roeckner@math.uni-bielefeld.de}}

\date{\today}

\numberwithin{equation}{section}

\begin{document}

\begin{abstract} 
We consider a class of infinite-dimensional singular stochastic control problems. These can be thought of as spatial monotone follower problems and find applications in spatial models of production and climate transition. Let $(D,\mathcal{M},\mu)$ be a finite measure space and consider the Hilbert space
$H:=L^2(D,\mathcal{M},\mu; \mathbb{R})$. Let then $X$ be an $H$-valued stochastic process on a suitable complete probability space, whose evolution is determined through an SPDE driven by a self-adjoint linear operator $\mathcal{A}$ and affected by a cylindrical Brownian motion. The evolution of $X$ is controlled linearly via an $H$-valued control consisting of the direction and the intensity of action, a real-valued nondecreasing right-continuous stochastic process, adapted to the underlying filtration. The goal is to minimize a discounted convex cost-functional over an infinite time-horizon. By combining properties of semiconcave functions and techniques from viscosity theory, we first show that the value function of the problem $V$ is a {$C^{1,\mathrm{Lip}}(H)$}-viscosity solution to the corresponding dynamic programming equation, which here takes the form of a variational inequality with gradient constraint. Then, by allowing the decision maker to choose only the intensity of the control and requiring that the given control direction $\hat{n}$ is an eigenvector of the linear operator $\mathcal{A}$, we establish that the directional derivative $V_{\hat{n}}$ is of class $C^1(H)$, hence a second-order smooth-fit principle in the controlled direction holds for $V$. This result is obtained by exploiting a connection to optimal stopping and combining results and techniques from convex analysis and viscosity theory. 
 
\end{abstract}

\maketitle

\smallskip


\smallskip

{\textbf{MSC2020 subject classification}}: 93E20, 37L55, 35D40, 49J40, 60G40, 91B72.


\section{Introduction}
\label{introduction}

Singular control and optimal stopping problems arise frequently in Economics, Finance, Engineering, and related fields. Due to their inherent complexity, much analysis tends to focus on one-dimensional problems, where our understanding is relatively comprehensive. However, contemporary societal challenges present complex structures for singular control problems, that ask for a rigorous and sound mathematical basis. This paper introduces a framework for addressing singular control problems in the context of state processes governed by stochastic partial differential equations. By combining convex-analytic arguments and the theory of viscosity solutions, we show that the problem's value function $V$ is a $C^{1,Lip}(H)$-viscosity solution to the corresponding dynamic programming equation. Furthermore, by exploiting a connection to a suitable family of simpler optimal stopping problems, we are able to further enhance regularity and prove that a second-order smooth-fit principle holds for $V$. Finally, we discuss potential applications in fields such as energy economics or climate modeling. 

Let us describe the  class of infinite-dimensional singular stochastic control problems and our contributions more precisely. Let $(D,\mathcal{M},\mu)$ be a finite measure space, and consider the Hilbert space $H:=L^2(D,\mathcal{M},\mu; \mathbb{R})$. The state variable is described by a stochastic process $X$ with values in  $H$. Its evolution is determined by a stochastic partial differential equation (SPDE) driven by a self-adjoint linear operator $\mathcal{A}$ and   a cylindrical Brownian motion. Next to   technical requirements, we assume that the operator $\mathcal{A}$ generates a $C_0$-semigroup of positivity-preserving contractions. As a benchmark case, one can consider the sum of the Laplacian operator and a multiplicative operator of the form $-\delta x$, for $\delta>0$, representing a depreciation or dissipative term. The evolution of $X$ is controlled linearly through an $H$-valued control, encompassing both direction and intensity of action, the latter being a real-valued nondecreasing right-continuous stochastic process adapted to the underlying filtration. The objective is to minimize a discounted convex cost-functional of the form
\begin{equation}
    \label{eq:costJi-ntro}
\mathcal{J}(x;I):=\E \bigg[\int_{0^-}^{\infty}e^{-\rho t} \Big({G}(X_t^{x,I})\d t +\langle q, \d I_{t}\rangle_H \Big)\bigg], \quad (x,I) \in H \times \mathcal{I},
\end{equation}
where $X^{x,I}$ is the state process starting at $x$ and controlled via $I \in \mathcal{I}$ (cf.\ \eqref{set:S} below), $G$ is a running cost function (see Assumption \ref{ass:C} below), $q \in H$ is a (stritly positive) proportional cost of action, and  $\rho>0$ is an intertemporal discount rate. 
The problem under study can thus be  thought of as the infinite-dimensional version of the {monotone follower} problems addressed in \cite{K81, K83, {KaratzasShreve84}}, among others.

Our analysis begins by establishing preliminary regularity properties of the problem's value function, denoted as $V$. Specifically, assuming that the running cost function $G$ is convex and semiconcave (with respect to the norm of $H$), we demonstrate that these properties are inherited by $V$. Thus, by adapting results from \cite{CannarsaSinestrari} to our infinite-dimensional setting, we find that $V \in C^{1,\text{Lip}}(H)$.


We begin by deriving the dynamic programming equation associated with \eqref{eq:costJi-ntro} and show that $V$ is a $C^{1,\text{Lip}}(H)$-viscosity solution of this equation. 
The proof relies on Dynkin’s formula for test functions of semimartingales and on an alternative representation of $V$, obtained through a tailored application of the Radon–Nikodym theorem for vector-valued measures (see Lemma \ref{lemma:newcontrol} and \eqref{eq:costfunctbis} below). 
In particular, the proof of the supersolution property of $V$ introduces a new argument based on an inequality derived from Dynkin’s formula together with the dissipativity property of the operator $\mathcal{A}$. 
This approach also extends naturally to finite-dimensional settings, greatly simplifying the usual technical difficulties in proving the supersolution (or subsolution) property in minimization (or maximization) problems involving singular controls (see, e.g., \cite{Ch}, \cite[Ch.~VIII]{FlemingSoner}, \cite{HS}, and \cite{Ma}).

To further enhance the regularity of $V$, we introduce the assumption that the decision-maker can only control the intensity of action. The direction of action, denoted by $\hat{n} \in H$, is then taken to be an eigenvector of the operator $\mathcal{A}$. Under this requirement and further technical properties of $G$, we are able to show that the directional derivative of $V$ in the direction of control, $V_{\hat{n}}$, is such that $V_{\hat{n}} \in C^1(H)$. This result can be read as a second-order smooth fit property of $V$, a regularity result of particular relevance in singular stochastic control problems (see the discussion in Section \ref{cor:2ndsmooth} below). The aforementioned smooth-fit property is obtained by identifying $V_{\hat{n}}$ as the value function of an optimal stopping problem (in the spirit of the finite-finite dimensional contribution \cite{BK}) and subsequently examining the regularity of its (sub)gradient. In particular, under a suitable nondegeneracy condition on the Brownian noise, assuming that the directional derivative $G_{\hat{n}}$ is semiconcave, and combining arguments from viscosity theory and convex analysis, we are able to show that $V_{\hat{n}}$ is Fr\'echet differentiable at any $x\in H$ and that the gradient $DV_{\hat{n}} \in C(H;H)$. For further details, please refer to Proposition \ref{prop:improved}. 

In Section \ref{sec:applications}, we demonstrate the relevance of our framework in economic applications. For instance, we examine an irreversible investment problem in energy capacity and an energy balance climate model incorporating human impact. In the former, an energy producer seeks to maximize the net total expected surplus resulting from irreversible investments in energy production. In the latter, temperature is increased by human activities through carbon emissions and a social planner aims to minimize an intertemporal expected cost criterion, penalizing temperature deviations from an ideal level, such as pre-industrial temperatures.

Let us now discuss related literature and our contribution to it. The origin of singular stochastic control dates back to the early contributions by Bather and Chernoff \cite{BatherChernoff67} , and later by Bene\v{s}, Shepp, and Witsenhausen \cite{Benesetal} and Karatzas \cite{K81, K83}. Those seminal papers deal with one-dimensional problems of so-called \emph{monotone follower} type, in which a process with monotone paths (or, more generally, of bounded-variation) has to be chosen in order to track the evolution of a Brownian motion so that an expected cost criterion is minimized. Since then, the theory of singular stochastic control has attracted increasing attention, also boosted by its connection to optimal stopping (see \cite{BK}, \cite{BoetiusKohlmann}, and \cite{KaratzasShreve84}, just to cite a few) and its numerous applications in Economics and Finance. Among those, problems of optimal capacity expansion \cite{BK}, optimal investment with transaction costs \cite{SoSh}, optimal harvesting \cite{Al00}, and optimal dividends' distribution \cite{JS}.

For stationary one-dimensional problems, or for two-dimensional degenerate problems with a suitable structure \cite{Ferrari15,MZ07}, explicit solutions can be expected. Typically, these solutions are obtained through the guess-and-verify approach. This involves first determining a smooth solution to the problem's dynamic programming equation (in this case, a variational inequality with gradient constraints), and then verifying its optimality using a version of It\^o's formula. Additionally, an optimal control is determined as a byproduct of this analysis. This is given in terms of the solution to a Skorokhod reflection problem at the so-called free boundary, i.e.\ the topological boundary of the region in which the gradient constraint is not active (the so-called no-action or continuation region). 

For time-dependent problems or for stationary problems in dimension larger than one, the guess-and-verify approach is not feasible. This is because the dynamic programming equation now becomes a partial differential equation (PDE) with gradient constraints, for which explicit solutions are typically not available. As a consequence, direct probabilistic and analytical approaches are put in place in order to obtain regularity of the value function (typically under convexity requirements; see, e.g., \cite{Hynd1, Hynd2, Menaldi, SoSh2}) and, when possible, to characterize the optimal control as the minimal amount of effort needed to keep the underlying state process within the no-action region (see  \cite{DianettiFerrari}, \cite{Kruk}, and references therein). As a matter of fact, differently to before, in multi-dimensional settings, the free boundary is not explicit and constructing a solution to the related Skorokhod reflection problem is far from trivial. We refer to the introduction of \cite{DianettiFerrari} for a discussion on this aspect. The aforementioned challenges explain why the number of contributions on singular stochastic control problems in multi-dimensional settings is still very limited. 

The theory of regular stochastic control and of optimal stopping in infinite-dimensional (notably, Hilbert) spaces received a large attention in the last decades (see, e.g., the monography \cite{FGS} and the recent regularity results in \cite{defeo} for control problems, and \cite{BM}, \cite{CDA}, \cite{ChowMenaldi}, \cite{FO}, \cite{Fuhrman}, \cite{GS}, \cite{SwiechTexteira}, \cite{Zabczyk} for optimal stopping). As previously discussed, we contribute to that bunch of literature by providing the viscosity property and $C^1$-regularity (smooth-fit) of the value function of a class of  optimal stopping problems in Hilbert spaces. To the best of our knowledge, such a regularity result appears here for the first time, and it is therefore of independent interest.

On the other hand, the literature on singular stochastic control in infinite-dimensional spaces is very limited. The only three papers brought to our attention are \cite{Oksendal2} and \cite{Oksendal1}, motivated by optimal harvesting, and ours \cite{FFRR}. In \cite{Oksendal1} the problem is posed for a quite general controlled SPDE, which also enjoys a space-mean dependence in \cite{Oksendal2}. The authors establish a necessary Maximum Principle, which is also sufficient assuming the concavity of the Hamiltonian function pertaining to the control problem under consideration. However, despite their significant contributions, there appears to be a foundational concern when dealing with (singularly controlled) SPDEs, particularly regarding the existence of a solution and the application of It\^o's formula (refer to \cite{LR} for theory and results on SPDEs).
Specifically, it is important to notice that in infinite-dimensional singular (stochastic) control problems, the precise interpretation of the integral with respect to the vector measure represented by the control process - and thus the exact interpretation of the controlled state equation - poses a nuanced issue that warrants careful consideration. Finally, our previous work \cite{FFRR} derives necessary and sufficient conditions for a class of singular stochastic control problems on an abstract partially ordered infinite-dimensional space. The main differences with respect to the present work are in the framework, the methodology, and the nature of the results. In \cite{FFRR}, the controlled state process is fully degenerate and randomness comes into the problem only in a parametric form, thus making the underlying optimization problem not necessarily Markovian. Furthermore, the main result is obtained by the exclusive mean of convex analytic arguments, and no statement about the regularity of the value function is made. In this work, we deal with a singularly controlled SPDE and exploit the dynamic programming approach 
 together with viscosity theory and convex analysis in order to achieve regularity results on the problem's value function.
\vspace{0.25cm}

\textbf{Organization of the Paper.}\ The rest of the paper is organized as follows. In Section \ref{sec:setting} we provide the setting and introduce the problem. In Section \ref{sec:preliminaryvisc-V} we then consider the variational inequality associated to the problem and prove preliminary regularity and viscosity property of its value function $V$. Under a suitable requirement on the direction of action, in Section \ref{sec:connection} a connection to optimal stopping is derived and regularity of the optimal stopping problem's value function is proved. As a byproduct of that, in Section \ref{sec:mainres} a second-order smooth-fit property for $V$ is then obtained. Finally, Section \ref{sec:applications} proposes two applications in Economics, while Appendix \ref{sec:appendixA} provides an extended sketch of the proof of the Dynamic Programming Principle and Appendix \ref{sec:appendixB} some technical lemmata.


\section{Setting and Problem Formulation}
\label{sec:setting}

\subsection{Setting} 
Let $(D,\mathcal{M},\mu)$ be a finite standard Borel measure space and assume, without loss of generality for what follows, that  $\mu(D)=1$. Consider the separable Hilbert space 
$$H:=L^2(D,\mathcal{M},\mu; \mathbb{R}).$$ The dual $H^*$ is identified with $H$ via the classical Riesz representation of $H^*$. The nonnegative cone of $H$ is denoted by 
$$H_+:=\big\{x\in H: \ x \geq 0\big\}.$$
We denote by $\mathcal{L}(H)$ the space of linear bounded operators on $H$ and by {$\mathcal{L}^+(H)$} the subspace of positivity-preserving operators of $\mathcal{L}(H)$; i.e., ${P}\in \mathcal{L}^+(H)$ if
$$
x\in H_+  \ \Longrightarrow \  Px\in H_+.$$

Throughout the paper, we consider a linear operator $\mathcal{A}:\mathcal{D}(\mathcal{A})\subseteq H\to H$ satisfying the following standing requirements.
\begin{assumption}
\label{ass:A}
$\mathcal{A}$ is self-adjoint, closed, densely defined, and such that, for some $\delta>0$, we have  
$$\langle \A x, x \rangle_H \leq - \delta |x|^2_H, \ \ \forall x\in H.$$
\end{assumption}
In particular (see, e.g., \cite[Ch.\,II,\,Sec.\,3]{EN} and \cite[Ch.\,II-1, Sec.\,2.10.1]{BDDM}), under Assumption \ref{ass:A}, the operator $\mathcal{A}$ generates a $C_0$-semigroup of contractions $(e^{t\A})_{t\geq 0}\subseteq \mathcal{L}(H)$ and 
$$
|e^{t\A}|_{\mathcal{L}(H)}\leq   e^{-\delta t},\ \ \ \forall t\geq 0.
$$
Moreover, $0\in\varrho(\A)$ -- with $\varrho$ denoting the resolvent set -- so that $\A$ is invertible and $$\A^{-1}\in \mathcal{L}(H).$$
We also assume the following.
\begin{assumption}
\label{ass:A2}
The $C_0$-semigroup of contractions $(e^{t\A})_{t\geq 0}\subseteq \mathcal{L}(H)$ is positivity-preserving; that is, $(e^{t\A})_{t\geq 0}\subseteq \mathcal{L}^+(H)$.
\end{assumption}

\begin{remark}
\label{rem:positivitypres}
Sufficient conditions guaranteeing positivity of semigroups can be found, e.g., in \cite[Chap.\ C-II, Thm.\ 1.2, Thm.\ 1.8]{Arendt-etal} and \cite[Thm.\ 7.29 and Prop.\ 7.46]{Clement-etal} 
\end{remark}



Let us now come to the probabilistic structure of our setup.
We endow the time-interval $[0,\infty)$ with the Borel $\sigma$-algebra $\mathcal{B}([0,\infty))$.
Also, let $(\Omega,\mathcal{F},\F,\P)$ be a filtered probability space, with filtration $\F:=(\mathcal{F}_t)_{t\in[0,\infty)}$ satisfying the usual conditions, and let $W$ be a cylindrical Wiener process on $(\Omega,\mathcal{F},\F,\P)$, taking values in another Hilbert space $K$. Finally, for future use, we denote by $\mathcal{T}$ the set of all $\mathbb{F}$-stopping times.

In the following, all the relationships involving $\omega\in\Omega$ as hidden random parameter are intended to hold $\P$-almost surely. Also, in order to simplify the exposition, often we will not stress the explicit dependence of the involved random variables and processes with respect to $\omega\in \Omega$.  

Let $\Delta \subseteq H_+$ be a convex cone of $H_+$ and set
\begin{eqnarray}
\label{set:S0}
\mathcal{M}&:=& \big\{I: \Omega \times [0,\infty) \to H_+:\,\text{$I_{\cdot}$ is}\,\,\F-\text{adapted and such that}\,\, t \mapsto I_t \nonumber \\
&& \hspace{1.5cm} \text{is c\`adl\`ag and with}\,\, I_{t} - I_{s-} \in \Delta \,\, \ \forall s,t \in [0,\infty)\,\,\text{such that}\,\,t\geq s \big\}.
\end{eqnarray}
Notice that, since any $I \in \mathcal{M}$ takes values in $H_+$, right-continuity is intended in the norm of $H$. In the following, we set $I_{0^-}:=\textbf{0}\in H_+$ for any $I \in \mathcal{M}$ (see Remark \ref{rem:nu0} below).

Any given $I \in \mathcal{M}$ can be seen as a (random) countably additive vector measure $$I: \mathcal{B}([0,\infty)) \to H_+$$ of local finite variation,
defined as  
$$I([s,t]):= I_t - I_{s^-} \  \ \ \forall s,t\in [0,\infty), \ t\geq s.$$ 
We denote by $|I|$ the variation of $I$; it  is a nonnegative (optional random) measure on $([0,\infty),\mathcal{B}([0,\infty)))$ that, due to monotonicity of $I$, can be simply expressed as 
\begin{equation}
\label{eq:dabsI}
    |I|([s,t]) = |I_t - I_{s^-}|_H, \ \ \ \forall s,t\in [0,\infty), \ s\leq t.
\end{equation}

\begin{remark}
\label{rem:nu0}
By setting $I_{0^-}:=0$ for any $I \in \mathcal{M}$, we mean that we extend any $I \in \mathcal{M}$ by setting $I \equiv 0$ on $[-\varepsilon,0)$, for a given and fixed $\varepsilon>0$. In this way, the associated measures have a positive mass at initial time of size $I_0$. Notice that this is equivalent with identifying any control $I$ with a countably additive measure $I: \mathcal{B}([0,\infty)) \to [0,\infty)$ of local finite variation defined as $I((s,t]):= I_t - I_{s}$, for every $s,t\in [0,\infty)$, $s < t$, plus a Dirac-delta at time $0$ of amplitude $I_0$.
\end{remark}

Since $H$ is a reflexive Banach space, by \cite{DU}, Corollary 13 at p.\ 76 (see also Definition 3 at p.\ 61), there exists a  Bochner measurable function  $\hat \vartheta=\hat \vartheta(\omega):[0,\infty) \to H_+$ such that 
\begin{equation}
\label{DP}
\int_{[0,T]} |\hat \vartheta_t|_{H} \d|I|_t < \infty \ \ \ \ \forall T>0\quad \text{and} \quad \d I_{t}= \hat \vartheta_{t}\, \d|I|_{t} \ \ \ \forall t\geq 0.
\end{equation}
Notice that, seen as a stochastic process, $\hat \vartheta=(\hat \vartheta_t)_{t\geq  0}$ is $\mathbb{F}$-adapted, because so is $I$. Furthermore, given that the measures $I$ and $|I|$ are equivalent by \eqref{eq:dabsI}, one has 
\begin{equation}\label{eq:hattheta}
\hat{\vartheta}_t\neq \textbf{0} \ \ \ \mbox{for a.e.} \  t\geq 0.
\end{equation}
The process $\hat\vartheta$ is clearly unique up to $\P\times |I|-$null measure sets.

Then, for a given $H_+-$valued $\mathbb{F}$-adapted process ${f}:=(f_t)_{t\in [0,\infty)}$, recalling  \eqref{DP}, for any $t\in [0,\infty)$ we define 
\begin{eqnarray}
\label{def-int1}
& \displaystyle \int_{0^-}^t \langle f_s,\,\d I_s\rangle_H :=\int_{[0,t]} \langle f_s, \hat{\vartheta}_s\rangle_H \, \d |I|_s = \int_{[0,t]} \Big(\int_D f_s(\xi) \hat{\vartheta}_s(\xi) \mu(\d \xi) \Big) \d |I|_s \nonumber \\
&= \displaystyle \int_D \Big(\int_{[0,t]} f_s(\xi) \hat{\vartheta}_s(\xi) \d |I|_s \Big) \mu(\d \xi),
\end{eqnarray}
where the last step is possible due to Fubini-Tonelli's theorem. 
With regard to \eqref{DP}, we also set
\begin{equation}
\label{eq:I-conv}
\int_{0^-}^{t}e^{(t-s)\A} \d I_{s} := \int_{0^-}^{t}e^{(t-s)\A}\hat{\vartheta}_{s} \d |I|_{s}, \quad t \geq 0. 
\end{equation}

Thanks to \eqref{eq:I-conv}, for any given $I \in \mathcal{M}$, we can then introduce the singularly continuous controlled dynamics
\begin{equation}
\label{eq:state}
\d X^{x,I}_{t}=\mathcal{A}X^{x,I}_{t}\d t+\sigma \d W_{t} + \d I_{t}, \quad t \geq 0,  \ \ \ X_{0^{-}}^{x,I}=x \in H,
\end{equation}
and define the unique mild solution to \eqref{eq:state} as
\begin{equation}
\label{mild}
X_{t}^{x,I}=e^{t\A}x+W_{t}^{\A,\sigma}+ \int_{0^-}^{t}e^{(t-s)\A}\d I_{s}, \quad t\geq 0.
\end{equation}
Here, 
\begin{equation}
\label{eq:stochintegr}
W_{t}^{\A,\sigma}:= \int_{0}^t{e}^{(t-s)\A}\sigma \d W_{s}, \quad t \geq0,
\end{equation}
with $\sigma$ satisfying the following standing condition. 
\begin{assumption}
\label{ass:sigma}
$\sigma\in\mathcal{L}_2(K;H)$, where $\mathcal{L}_2(K;H)$ denotes the space of Hilbert-Schmidt operators from $K$ to $H$.
\end{assumption}
We denote by $\mathcal{L}_1(H)$ the set of  nuclear operators on $H$, i.e.  operators $Q\in\mathcal{L}(H)$ such that 
\begin{equation}\label{normL1}
|Q|_{\mathcal{L}_{1}(H)}:=\mbox{Tr}[|Q|]:=\sum_{k=0}^{\infty}\langle \sqrt {Q^{*}Q}e_{k},e_{k}\rangle_{H}<\infty,
\end{equation}
where $(e_{k})$ is any orthonormal basis of $H$. 
As well known, the above quantity does not depend on the choice of the basis and $\mathcal{L}_1$ is a Banach space endowed by the norm defined in \eqref{normL1} .
We then have, under Assumption \ref{ass:sigma}, that
$
\sigma\sigma^{*}\in\mathcal{L}_{1}(H).
$

Notice that Assumptions \ref{ass:A} and \ref{ass:sigma} imply that the stochastic convolution \eqref{eq:stochintegr} is well defined and continuous (see \cite[Ch.\,5]{DPZ}). For future use, we also note that, because of Assumptions \ref{ass:A} and \ref{ass:sigma}, for all $m\in [1,\infty)$ one has for some $\overline{c}_{m}>0$ (see \cite{Hausenblas-Seidler})
\begin{equation}
\label{estimatesup-W}
\E\left[\sup_{t\geq 0}|W^{\mathcal{A},\sigma}_{t}|^m_{H}\right] \leq \overline{c}_{m},
\end{equation}
which, denoting the mild solution to \eqref{eq:state} when $I$ is the null control by $X_{t}^{x,0}$, implies
\begin{equation}
\label{estimatesup}
\E\left[\sup_{t\geq 0}|X^{x,0}_{t}|^m_{H}\right]\leq c_{m}(1+|x|^m_{H}), \ \ \ \ \forall x\in H,
\end{equation}
for some other constant $c_{m}>0$.
\medskip

\begin{remark}
\label{rem:reaction-diff}
The Ornstein–Uhlenbeck setting plays a fundamental role in guaranteeing -- in a direct but still not immediate way -- the well-posedness of the state equation in the infinite-dimensional context, thereby providing a natural starting point for the analysis of the control problem. Its linear structure makes it possible to rigorously define the mild solution of the underlying controlled stochastic evolution equation and to derive the necessary a priori estimates on the state process. Moreover, this framework allows for the explicit derivation of gradient estimates and facilitates the application of semi-concavity and regularization techniques within a Hilbert space setting.

Extending the analysis to the relevant case of nonlinear dynamics  introduces substantial issues even at the preliminary stage of studying the well-posedness of the controlled state equation, due to the presence of the singular control. We leave a detailed study of this interesting class of problems for future research. 
\end{remark}

\subsection{Problem formulation}
We now move on by introducing the infinite-dimensional singular stochastic control problem which is the object of our study. Let $G:H \to \mathbb{R}$ be a running cost function, satisfying the following requirements.
\begin{assumption}
\label{ass:C}
\begin{enumerate}[(i)]
\item[]
\item ${G}$ is convex and 
there exist $c_o, \kappa_1, \kappa_2 >0$ such that
$$
\kappa_1 |x|_{H}^{2} - \kappa_2 \leq {G}(x)\leq c_{o}(1+|x|_{H}^{2});
$$ 
\item $G$ is semiconcave with semiconcavity constant $c_1>0$; that is, there exists $c_1>0$ such that
$$
\lambda G(x)+(1-\lambda) G(y) - G(\lambda x+(1-\lambda) y)\leq \frac{c_1}{2}\lambda(1-\lambda)|x-y|_{H}^{2}, \ \ \forall x,y\in H, \ \lambda\in[0,1].
$$
\end{enumerate}
\end{assumption}

Notice that by \cite{LasryLions} (see end of Page 265 therein) one has that $G \in C^{1, \text{Lip}}(H)$.

\begin{remark}
\label{rem:AssG}
\begin{enumerate}[(i)]
\item 
Benchmark examples satisfying Assumption \ref{ass:C} are the quadratic cost function
$$G(x)=\frac{1}{2}|x-\overline{x}|^2_H,     \ \ \ \ x\in H,$$  for some target level $\overline{x}\in H$, as well as 
$$
{G}(x)=\frac{1}{2}\langle x,h\rangle_{H}^2, \quad \text{or} \quad {G}(x)=\frac{1}{2}\langle Q x, x\rangle_{H}, \quad x \in H,
$$
with $h \in H$, and with $Q$ being positive semidefinite and symmetric.
\item Assumption \ref{ass:C}(ii) amounts to requiring semiconcavity with a linear modulus. One could weaken this condition by assuming that, for some $\alpha \in (0,1]$, 
\begin{equation}
    \label{eq:alphasemiconcave}
\lambda G(x) + (1-\lambda) G(y) - G(\lambda x + (1-\lambda) y) 
\leq \frac{c_1}{2} \lambda(1-\lambda) |x - y|_{H}^{1+\alpha}, 
\end{equation}
for all $x,y \in H$ and for all $\lambda \in [0,1]$. Under this weaker assumption, all the results in this paper remain valid, after making the corresponding adjustments. As a matter of fact, by mimicking the arguments developed in the finite-dimensional setting in \cite[Th.\,3.3.7]{CannarsaSinestrari}, one can still prove that a convex function satisfying \eqref{eq:alphasemiconcave} belongs to the class $C^{1,\alpha}_{loc}(H)$.
\end{enumerate}

\end{remark}

For a discount rate $\rho>0$ and for $q\in H_+$ such that $q \geq q_o \mathbf{1}$ for some $q_o>0$ (being $\mathbf{1} \in H$ the constant function of $H$ identically equal to $1$), recalling \eqref{def-int1} we introduce the expected cost functional 
\begin{equation}
\label{eq:costfunct}
\mathcal{J}(x;I):=\E \bigg[\int_{0^-}^{\infty}e^{-\rho t} \Big({G}(X_t^{x,I})\d t +\langle q, \d I_{t}\rangle_H \Big)\bigg], \quad (x,I) \in H \times \mathcal{I}.
\end{equation}
Here, the the class of admissible controls $\mathcal{I}$ is defined as (cf.\ \eqref{set:S0})

\begin{eqnarray}
\label{set:S}
\mathcal{I} & := & \Big\{I\in \mathcal{M}:\,\,\,\E\bigg[\Big|\int_{0^-}^T |\hat{\vartheta}_s|_H\, \d |I|_s\Big|^2\bigg] + \E\bigg[\int_{0^-}^T \frac{|\hat{\vartheta}_s|^2_H}{\langle q, \hat{\vartheta}_s \rangle}\, \d |I|_s\bigg] < \infty \,\,\, \forall T>0\Big\}.
\end{eqnarray}

The infinite-dimensional singular stochastic control problem under study is then
\begin{equation}
\label{eq:V}
V(x):=\inf_{I\in\mathcal{I}}\mathcal{J}(x;I), \quad x \in H.
\end{equation}

\begin{remark}
\label{rem:integrableI}
Notice that the integrability condition in \eqref{set:S} is not required for the well posedness of \eqref{eq:costfunct}, but it will be needed in the next section for the proof of the viscosity property of $V$.
\end{remark}

Given the structure of the cost functional \eqref{eq:costfunct}, it is convenient to rewrite the decomposition \eqref{DP} in a tailored way based on the instantaneous cost of control $\langle q, \d I_t\rangle_H$. To that end, recall that $\Delta \subseteq H_+$ is a convex cone of $H_+$ (cf.\ \eqref{set:S}, consider the convex set 
\begin{equation}
\label{setTHETA}
  \Theta:=\big\{\theta\in \Delta: \ \langle q, \theta \rangle_{H}=1\big\}, 
\end{equation}
and define
\begin{eqnarray}
\label{set:S-new}
\mathcal{S}&:=& \{\nu: \Omega \times [0,\infty) \to [0,\infty):\,\text{$\nu_{\cdot}$ is}\,\,\F-\text{adapted and such that}\,\, t \mapsto \nu_t \nonumber \\
&& \hspace{1.5cm} \text{is \ c\`adl\`ag and nondecreasing}\}.
\end{eqnarray}
In the sequel, we set $\nu_{0^-}:=0$ for any $\nu \in \mathcal{S}$ (see also Remark \ref{rem:nu0}).
Then, define 
\begin{eqnarray}
    \mathcal{I}_0& :=& \Big\{(\vartheta, \nu):\Omega \times [0,\infty) \to \Theta \times [0,\infty): \ \vartheta_\cdot \ \mbox{is} \ \F-\mbox{adapted},\,\, \nu_\cdot\in \mathcal{S}\,\,\, \text{and} \nonumber \\
    && \E\bigg[\Big|\int_{0^-}^T |{\vartheta}_s|_H\, \d \nu_s\Big|^2\bigg] +  \E\bigg[\int_{0^-}^T |{\vartheta}_s|^2_H\, \d \nu_s\bigg] < \infty \,\,\,\text{for}\,\, \forall T>0 \Big\}.
    \label{setI0}
\end{eqnarray}

\begin{lemma}
\label{lemma:newcontrol}
For each $I\in\mathcal{I}$,  there exists a 
couple $(\vartheta,\nu)\in \mathcal{I}_{0}$, with $\nu\sim |I|$, such that 
\begin{equation}
\label{DP1}  
\d I_{t}=  \vartheta_{t}\, \d\nu_{t}, \ \ \ \ \forall t\geq 0.
\end{equation}
This couple is unique in the following sense: the optional random measure $\nu$ is unique and $\vartheta$ is unique up to $\P \otimes \nu-$null measure sets. 
\end{lemma}

\begin{proof}
Due to \eqref{DP}, positivity of $\hat{\vartheta}$ and the fact that $q \geq q_o \mathbf{1}$, for some $q_o>0$, we can write for any $t\geq0$
$$
  \d I_{t}= \hat\vartheta_{t}\d|I|_{t} =  \frac{\hat\vartheta_{t}}{\langle q, \hat{\vartheta}_{t}\rangle_{H}}
  {\langle q, \hat{\vartheta}_{t} \rangle_{H}} \d|I|_{t} = \vartheta_{t}\, \d\nu_{t},
$$
where
$$
\vartheta_{t} := \frac{\hat{\vartheta}_t}{\langle q, \hat{\vartheta}_{t}\rangle_{H}},  \ \ \text{and} \ \  \d\nu_{t}:=  {\langle q, \hat{\vartheta}_{t}\rangle_{H}} \d|I|_{t}.
$$
Recalling \eqref{set:S}, one clearly has that the integrability conditions required in \eqref{setI0} are met, because of the previous definitions of $\vartheta$ and $\d\nu$, and because $I\in\mathcal{I}$.
This shows the first part of the claim.
\smallskip

Let us prove uniqueness. Assume that 
$$
  \d I_{t}=\vartheta^{(1)}_{t}\d\nu^{(1)}_{t} =  \vartheta^{(2)}_{t}\d \nu^{(2)}_{t}.
$$
Then, for all $0\leq a\leq b$, 
$$
\int_{[a,b]} \langle q, \vartheta^{(1)}_{t}\rangle_{H}\d\nu^{(1)}_{t} = \int_{[a,b]} \langle q, \vartheta^{(2)}_{t} \rangle_{H}\d\nu^{(2)}_{t}, 
$$
implying, by definition of $\Theta$,
$$
\int_{[a,b]} \d\nu^{(1)}_{t} = \int_{[a,b]} \d\nu^{(2)}_{t}.
$$
Hence, $\nu^{(1)}=\nu^{(2)}=:\nu$. 
Then, for all $0\leq a\leq b$, 
$$
\int_{[a,b]} \d I_{t}= \int_{[a,b]}  \vartheta^{(1)}_{t} \d\nu_{t} = \int_{[a,b]}  \vartheta^{(2)}_{t}\d\nu_{t}, 
$$
implying $\vartheta^{(1)}=\vartheta^{(2)}$ up to $\P \otimes \nu$-null measure sets.
\end{proof}
Thanks to Lemma \ref{lemma:newcontrol}, we may identify $\mathcal{I}$ with $\mathcal{I}_{0}$ (cf.\ \eqref{set:S} and \eqref{setI0}, respectively). Hence, hereafter, with a slight abuse of notation, we will often identify elements of the above sets.
The cost functional \eqref{eq:costfunctbis} then rewrites as 
\begin{equation}
\label{eq:costfunctbis}
\mathcal{J}(x;I)=\E \bigg[\int_{0^-}^{\infty}e^{-\rho t} \Big({G}(X_t^{x,I})\d t +\d \nu_{t} \Big)\bigg], \ \ \ \ \ \ x \in H, \  {I}=(\vartheta,\nu)\in\mathcal{I}_0,
\end{equation}
and the value function as
\begin{equation}
\label{eq:newV}
V(x)=\inf_{(\vartheta,\nu)\in\mathcal{I}_0}\mathcal{J}(x;I), \quad x \in H.
\end{equation}

In the next Section \ref{sec:preliminaryvisc-V}, we will make use of both the equivalent representations \eqref{eq:V} and \eqref{eq:newV}. In particular, preliminary regularity properties of $V$ (see Section \ref{sec:preliminary-V} below) as well as Propositions \ref{Dyn1} and \ref{Dyn2} in Section \ref{sec:viscosity-V} will be shown by using \eqref{eq:V}, while the viscosity property of $V$ will be proved through \eqref{eq:newV} (see Proposition \ref{DPPOC} and Theorem \ref{thm:Viscous} in Section \ref{sec:viscosity-V}).


\section{Regularity and Viscosity Property of $V$}
\label{sec:preliminaryvisc-V}

In this section we first show via direct convex-analytic arguments that $V$ is convex, has subquadratic growth and it is such that $V\in C^{1,\text{Lip}}(H)$. Then, we prove that $V$ is a viscosity solution to the associated Hamilton-Jacobi-Bellman equation, which in the present setting takes the form of a variational inequality with gradient constraint.

\subsection{Preliminary properties of $V$}
\label{sec:preliminary-V}

Here we provide some a priori regularity properties of the value function $V: H \to \mathbb{R}$ as in \eqref{eq:V}. 
\begin{proposition}
\label{prop:regularityV}
\begin{enumerate}[(i)]
\item[]
\item 
$V$ is convex;
\item There exists $\hat{c}_{o}>0$ such that, for $\kappa_2>0$ as in Assumption \ref{ass:C},
$$
-\kappa_2 \leq V(x)\leq \hat{c}_{o}(1+|x|^{2}_{H});$$
\item $V$ is locally Lipschitz;
 \item $V$ is semiconcave with semiconcavity constant $\hat{c}_1$; that is, there exists $\hat{c}_1>0$ such that
$$
\lambda V(x)+(1-\lambda) V(y) -V(\lambda x+(1-\lambda) y)\leq \frac{\hat{c}_1}{2}\lambda(1-\lambda)|x-y|_{H}^{2}, \ \ \forall x,y\in H, \ \lambda\in[0,1];
$$
\item $V\in C^{1,\text{Lip}}(H)$.
\end{enumerate}
\end{proposition}
\begin{proof}
We prove each item separately.
\vspace{0.15cm}

\emph{Proof of (i)}. {{For $i=1,2$, let $x_{i}\in H$ and let $I^{(i)}$ be $\varepsilon$-optimal for the initial data $x_{i}$, for some $\varepsilon>0$; that is,
$$
V(x_{i})+\varepsilon \geq \mathcal{J}(x_{i};I^{(i)}).
$$
For $\lambda\in[0,1]$, define
$$
x_{\lambda}:=\lambda x_{1}+(1-\lambda)x_{2}, \ \ \ \ \ I^{(\lambda)}:=\lambda I^{(1)}+(1-\lambda)I^{(2)}.
$$
 Given that the mapping $(x,I) \mapsto X^{x,I}$ is linear (see \eqref{eq:stochintegr} and \eqref{mild}), we have
 $$
 X^{x_{\lambda}, I^{(\lambda)}}=\lambda X^{x_1, I^{(1)}}+(1-\lambda)X^{x_2, I^{(2)}}.
 $$
 Then, using 
 the convexity of $G$, we write from \eqref{eq:V}
 \begin{align*}
&V(x_{\lambda})\leq \mathcal{J}(x_{\lambda};I^{(\lambda)})
=\E \bigg[\int_{0^-}^{\infty}e^{-\rho t} \Big({G}(X_t^{x_{\lambda},I^{(\lambda)}})\d t +\langle q, \d I^{(\lambda)}_{t}\rangle_H \Big)\bigg]\\
& \leq \lambda \E \bigg[\int_{0^-}^{\infty}e^{-\rho t} \Big({G}(X_t^{x_{1},I^{(1)}})\d t +\langle q, \d I^{(1)}_{t}\rangle_H \Big)\bigg]+(1-\lambda) \E \bigg[\int_{0^-}^{\infty}e^{-\rho t} \Big({G}(X_t^{x_{2},I^{(2)}})\d t +\langle q, \d I^{(2)}_{t}\rangle_H \Big)\bigg]\\
& = \lambda\mathcal{J}(x_{1};I^{(1)})+(1-\lambda) \mathcal{J}(x_{2};I^{(2)})\leq \lambda V(x_{1})+(1-\lambda) V(x_{2}) + \varepsilon.
 \end{align*} 
 By arbitrariness of $\varepsilon$, the claim follows.}}
\vspace{0.15cm}

\emph{Proof of (ii)}. The bound from below is immediate given that $G\geq -\kappa_2$ on $H$, $q \in H_+$, and $I\in \mathcal{I}$. As for the bound from above, recall that $X^{x,0}$ denote the uncontrolled mild solution to \eqref{eq:state}. Then, by Assumption \ref{ass:C}(ii), \eqref{eq:stochintegr} and the fact that $|e^{t\mathcal{A}}|_{\mathcal{L}(H)}\leq e^{-\delta t}|x|_H$, we can write 

\begin{align*}
V(x)&\leq \mathcal{J}(x;0)=\E\left[\int_{0}^\infty e^{-\rho t} {G} (X_{t}^{x,0})\d t\right] \leq c_{0} \int_{0}^\infty e^{-\rho t}\left(1+\E[|X_{t}^{x,0}|_{H}^{2}]\d t\right)\\
&\leq {c_{0}}\left(\frac{1}{\rho}+2\int_{0}^\infty e^{-\rho t} \left(|e^{t\A}x|_{H}^{2}+\E[|W_{t}^{\sigma,\A}|_{H}^{2}]\right)\d t\right)\\
&\leq {c_{0}}\left(\frac{1}{\rho}+2\int_{0}^\infty e^{-(\rho+2\delta) t} |x|_{H}^{2} \d t + 2\frac{1}{\rho} \overline{c}_{2}\right)\\
&= {c_{0}}\left(\frac{1}{\rho}\big(1 + 2\overline{c}_{2}\big) + \frac{2}{\rho + 2\delta} |x|_{H}^{2}\right),
\end{align*}
and the claim is proved.
\vspace{0.15cm}

\emph{Proof of (iii)}. It follows from the previous items and, e.g., \cite[Cor.\,2.4, p.\,12]{EkelandTemam}.
\vspace{0.15cm}

\emph{Proof of (iv)}. Let  $x,y\in H$ and $\lambda\in [0,1]$, and take an $\varepsilon$-optimal control $I^\varepsilon\in\mathcal{I}$ for $\lambda x+(1-\lambda)y$.
Since the state equation is affine (cf.\ \eqref{mild}), we have
$$
X^{\lambda x+(1-\lambda)y, I^\varepsilon}=\lambda X^{x, I^\varepsilon}+(1-\lambda) X^{y, I^\varepsilon},
$$
and the semiconcavity of $G$ allows to write (cf.\ \eqref{eq:costfunct})
\begin{align*}
&\lambda V(x)+(1-\lambda)V(y)-V(\lambda x+(1-\lambda) y)\\&
\leq \lambda \mathcal{J}(x;I^\varepsilon)+(1-\lambda)\mathcal{J}(y;I^\varepsilon)-J(\lambda x+(1-\lambda)y;I^\varepsilon)+\varepsilon\\
& = \E\left[\int_0^\infty e^{-\rho t}\left(\lambda G(X_t^{x,I^\varepsilon})+(1-\lambda)G(X^{y,I^\varepsilon}_t)- G(\lambda X^{x,I^\varepsilon}_t+(1-\lambda)X^{y,I^\varepsilon}_t \right)\d t\right] +\varepsilon \\
& = \E\left[\int_0^\infty e^{-\rho t} \frac{c_1}{2} \lambda(1-\lambda) |X_t^{x,I^\varepsilon}-X^{y,I^\varepsilon}_t|^2\d t\right]+\varepsilon
= \frac{c_1}{2}  \lambda(1-\lambda) \int_0^\infty e^{-\rho t} \big| e^{t\mathcal{A}}(x-y)\big|_H^2\d t+\varepsilon\\
& \leq \frac{c_1}{2}  \lambda(1-\lambda) \int_0^\infty e^{-\rho t} e^{-2\delta t}|x-y|_{H}^2\d t+\varepsilon
=\frac{c_1}{2(\rho+2\delta)}\lambda(1-\lambda)|x-y|_{H}^2+\varepsilon.
\end{align*}
By arbitrariness of $\varepsilon$, the claim follows.
\vspace{0.15cm}

\emph{Proof of (v)}. This follows by \cite{LasryLions}, end of page 265.
\end{proof}

Note that, since $V$ is finite on $H$ and $G$ is uniformly bounded from below on $H$, we can restrict the optimization \eqref{eq:newV} to the class of controls
\begin{equation}
    \label{eq:admissiblesetfinal}
\hat{\mathcal{I}}_{0}:= \left\{I=(\vartheta,\nu)\in\mathcal{I}_{0}: \  \E\bigg[\int_{0^-}^{\infty} e^{-\rho t}\d \nu_{t}\bigg]<\infty\right\};
\end{equation}
that is, we have 
\begin{equation}
\label{eq:newnewV}
V(x)=\hat V (x):=\inf_{(\vartheta,\nu)\in\hat{\mathcal{I}}_0}\mathcal{J}(x;I), \quad x \in H.
\end{equation}


\subsection{Dynamic programming equation and viscosity solutions}
\label{sec:viscosity-V}

In this section we show that $V$ is a viscosity solution to the dynamic programming equation associated to \eqref{eq:newV}.
To that end, we consider the variational inequality with gradient constraint
\begin{equation}\label{VIOC}
\max\Big\{(\rho-\mathcal{G})v(x)-{G}(x),\  \sup_{\theta\in\Theta}\big\{- \langle D{v}(x),\theta\rangle_{H}-1\big\} \Big\}=0, \quad x \in H,
\end{equation}
where $\Theta$ is as in \eqref{setTHETA} and where we have defined the second-order differential operator $\mathcal{G}$, formally associated to the process $X^{x,0}$ and acting on sufficiently smooth functions $v:H\to \R$, as
\begin{equation}
\label{eq:generator}
[\mathcal{G}v](x):= \langle \mathcal{A}x,Dv (x)\rangle_H +\frac{1}{2} \mbox{Tr}\left[\sigma\sigma^{*}D^{2}v(x)\right].
\end{equation}

Recall that $\mathcal{D}(\A) \subseteq H$ denotes the domain of the operator $\A$ and endow it with the graph norm $|\cdot|_{\mathcal{D}(\mathcal{A})}$; that is,
\begin{equation}
\label{eq:graphnorm}
|x|^2_{\mathcal{D}(\mathcal{A})}:= {|x|^2_H + |\mathcal{A}x|^2_H}, \ \ \ \ \ x\in \mathcal{D}(\A).
\end{equation}

\subsubsection{Dynkin's formulae}
In order to provide the definition of viscosity solution to \eqref{VIOC}, we first introduce the suitable class of test functions, and we then prove  Dynkin's formulae for those functions. 
Let us start by  introducing the set of functions
\begin{align}
\label{ClassX}
\mathcal{X}:=&\Big\{\varphi\in C^{2}(H):  \ D\varphi\in C(H;\mathcal{D}(\A)), \ \  \sigma\sigma^{*} D^2\varphi\in C(H;\mathcal{L}_{1}(H)), \ 
 \nonumber \\
& \ \ \ \ \  \ \ \ \ \varphi(x)\leq \hat{C} (1+|x|_{H}^{2}), \  |D\varphi(x)|_{\mathcal{D}(A)}\leq \hat{C}(1+|x|_{H}), \ |\sigma\sigma^{*}D^{2}\varphi(x)|_{\mathcal{L}_{1}(H)}\leq \hat C
\Big\}.
\end{align}

 In the subsequent analysis we shall use that any $I\in \mathcal{I}$ is identified with an element $(\vartheta, \nu) \in \mathcal{I}_0$ (see Lemma \ref{lemma:newcontrol} and the discussion afterwards) and that the optimization in \eqref{eq:newV} can be actually performed over $\hat{\mathcal{I}}_0$ (cf.\ \eqref{eq:admissiblesetfinal}).
Also, 
 in the rest of this paper, for any measurable function $f:H \to \mathbb{R}$, we set, for any $\tau\in \mathcal{T}$ and any $I\in \mathcal{I}$,
\begin{equation}
\label{eq:convention}
e^{-\rho\tau} f(X^{x,I}_{\tau}) := \limsup_{t \to \infty}e^{-\rho t} f(X^{x,I}_{t}) \quad \text{on} \,\, \{\tau=+\infty\}.
\end{equation}

We then have the next Dynkin formula for functions belonging to $\mathcal{X}$.

\begin{proposition}
\label{Dyn1} 
Let $x\in H$, $\varphi\in\mathcal{X}$, $I\in \hat{\mathcal{I}}_0$. Let 
 $\tau \in \mathcal{T}$ be bounded and such that there exists $M>0$ such that $\sup_{0 \leq t \leq \tau}|X^{x,I}_{t}| \leq M$. Then, the following Dynkin formula holds true: 
\begin{equation}
\label{Dynkin}
\E\left[e^{-\rho \tau}\varphi(X^{x,I}_{\tau})\right]=\varphi(x)+\E\left[\int_{0}^{\tau} e^{-\rho t}[(\mathcal{G}-\rho )\varphi](X^{x,I}_{t})\d t+ \int_{0^-}^{\tau}e^{-\rho t}\langle D\varphi(X^{x,I}_{t}), \vartheta_{t}\rangle_{H}\d \nu_{t}\right].
\end{equation}
\end{proposition}
\color{black}
\begin{proof}
Let $(\mathcal{A}_{n})_{n\in \mathbb{N}}$ be the Yosida approximants of $\mathcal{A}$ (see, e.g., \cite[Eq.\ (3.7), Ch.\ II]{EN}), and let $X^{n;x,I}$ denote the solution to 
$$
\d X^{n;x,I}_{t}=\mathcal{A}_{n} X^{n;x,I}_{t}\d t+\sigma \d W_{t}+\d I_{t}, \ \ \ X_{0^{-}}=x \in H; 
$$
that is (cf.\ \eqref{mild}),
$$
X_{t}^{n;x,I}=e^{t\A_{n}}x+W_{t}^{\A_{n},\sigma}+ \int_{0^-}^{t}e^{(t-s)\A_{n}}\d I_{s},\quad t\geq 0.
$$
Also, set
$$
[\mathcal{G}_{n}\varphi](x):= \langle \mathcal{A}_nx,D\varphi (x)\rangle+\frac{1}{2} \mbox{Tr}\left[\sigma\sigma^{*}D^{2}\varphi(x)\right].
$$

Thanks to \cite[Lemma 3.4(ii), Ch.\ 2]{EN} and due to the fact that $\mathcal{A}$ is dissipative (cf.\ Assumption \ref{ass:A}), we then have, 
\begin{equation}
\label{convAstar}
\mathcal{A}_{n}y \to \mathcal{A}y \ \ \mbox{as} \  n\uparrow \infty  \ \ \mbox{and} \ \ \ \sup_{n\in\mathbb{N}}|\mathcal{A}_ny|_H\leq |\mathcal{A}y|_H \quad \forall y \in \mathcal{D}(\A).
\end{equation}
From the second claim of \eqref{convAstar}, it follows that 
\begin{equation}\label{pippo}
\sup_{n\in\mathbb{N}}|\mathcal{A}_ny|_H\leq |y|_{\mathcal{D}(A)}, \ \ \ \forall y\in \mathcal{D}(\mathcal{A}).
\end{equation}
Recalling that $\tau$ is bounded, so that $\tau\leq T<\infty$ for some $T>0$, we may now use  
  \cite[Thm.\ 27.2]{Metivier}, upon noticing that the stochastic integral 
  $$
  \int_{0}^{\tau} \langle \mathcal{A}_{n} X_{t}^{n;x,I},D\varphi(X_{t}^{n;x,I}) \rangle_{H}\d W_{t}
  $$
  vanishes in expectations by Lemma \ref{lem:convdom}(i) and by the growth assumption on  $D\varphi$.
  We therefore obtain the following Dynkin's formula
 : \begin{align}
\label{Dynkin:yos}
& \E[e^{-\rho \tau}\varphi(X^{n;x,I}_{\tau})]=\varphi(x)+\E\bigg[\int_{0}^{\tau} e^{-\rho t}\left[[(\mathcal{G}_{n}-\rho )\varphi](X^{n;x,I}_{t})\d t+\langle D\varphi(X^{n;x,I}_{t}),  \vartheta_t\rangle_{H}\d\nu_{t}\right]\bigg].
\end{align}

We now aim at taking $n\to\infty$ in \eqref{Dynkin:yos} to get the claim.
By Lemma \ref{lem:convdom}(ii) (see also \cite[Prop.\ 1.132]{FGS} for the regular control case and references therein), it holds
that $X_{t}^{n;x,I} \to X_{t}^{x,I}$ in probability for every $t\in[0,T]$ as $n \to \infty$. The only convergences that require some attention are 
\begin{equation}
\label{eq:nontrivialconv}
\lim_{n\to \infty}\E\left[\int_0^{\tau} e^{-\rho t}\langle  X^{n;x,I}_t, \A_{n}D \varphi(X^{n;x,I}_{t})\rangle_{H} \d t\right] = \E\left[\int_0^{\tau} e^{-\rho t}\langle X^{x,I}_{t}, {\A} D\varphi (X^{x,I}_t)\rangle_{H} \d t\right]
\end{equation}
and
\begin{equation}
\label{eq:nontrivialconv-1}
\lim_{n\to \infty}\E\bigg[\int_{0^-}^{\tau} e^{-\rho t}\langle D\varphi(X^{n;x,I}_{t}),  \vartheta_{t} \rangle_{H}\d \nu_{t}\bigg] =  \E\bigg[\int_{0^-}^{\tau} e^{-\rho t}\langle D\varphi(X^{x,I}_{t}),  \vartheta_{t}\rangle_{H}\d\nu_{t}\bigg].
\end{equation}
All the other convergences can be performed by using arguments as in the proof of \cite[Prop.\ 1.164]{FGS} (in the case of regular controls). 

As for \eqref{eq:nontrivialconv-1}, recalling \eqref{ClassX} one has 
\begin{eqnarray*}
&& \Big|\mathds{1}_{[0,\tau]}(t) e^{-\rho t} \langle D\varphi(X^{n;x,I}_{t}),  \vartheta_{t} \rangle_{H}\Big| \leq  \widehat{C}(1+ |X^{n;x,I}_{t}|_{H}) |{\vartheta}_t|_H
\end{eqnarray*}
and because $\tau \in \mathcal{T}$ is bounded (say, by some $T>0$) we have for $q\in(1,2)$, upon using H\"older's inequality,
\begin{align*}
& \E\bigg[\int_{0^-}^{T} \Big|\mathds{1}_{[0,\tau]}(t) e^{-\rho t}\langle D\varphi(X^{n;x,I}_{t}),  \vartheta_{t} \rangle_{H}\Big|^{q}\, \d \nu_{t}\bigg] \leq \widehat{C}\E\bigg[\int_{0^-}^T (1+ |X^{n;x,I}_t|^{q}) |\vartheta_{t}|^q_H \,\d\nu_t\bigg], \nonumber \\
& \leq \widehat{C}\E\bigg[\sup_{0\leq t \leq T} (1+ |X^{n;x,I}_t|^{q}) \int_{0^-}^T |\vartheta_{t}|^q_H \d\nu_t\bigg] \nonumber \\
& \leq \widehat{C}\E\bigg[\sup_{0\leq t \leq T} (1+ |X^{n;x,I}_t|^{\frac{2q}{2-q}}) \bigg]^{\frac{2-q}{2}}\E\bigg[\int_{0^-}^T |\vartheta_{t}|^{2}_H \d\nu_t\bigg]^{\frac{q}{2}} < \infty
\end{align*}
where the finiteness of the last terms above is due to the fact that $(\vartheta, \nu) \in \mathcal{I}_0$ (cf.\ \eqref{setI0}) and Lemma \ref{lem:convdom}(i).
Hence, Vitali's convergence theorem implies \eqref{eq:nontrivialconv-1}, upon recalling that  $X_{t}^{n;x,I} \to X_{t}^{x,I}$ in probability for every $t\in[0,T]$ as $n \to \infty$.

We now move on by proving the validity of \eqref{eq:nontrivialconv}. Notice that, if $(x_{n}) \subseteq H$, $ (y_n)\subset \mathcal{D}(\mathcal{A})$, and $x \in H$, $y\in \mathcal{D}(\mathcal{A})$,  we have, using \eqref{pippo},
\begin{align}
\label{eq:stimadiffence}
|\langle x_n,\A_{n} y_n\rangle-\langle x,\A y\rangle_H| & \leq
|\langle (x_n-x),{\A}_n y_n\rangle_H|
+|\langle x,{\A}_n (y_n-y)\rangle_H| \nonumber +|\langle x,({\A}_n-{\A})y\rangle_H|\\& \leq
|x_n-x|_H |{\A}_n y_n|_H
+|x|_H |\A_n (y_n-y)|_H +|x|_H |({\A}_n-{\A})y|_H\nonumber\medskip\\
&\leq |x_n-x|_H\,| y_n|_{\mathcal{D}(\mathcal{A})}+|x|_H\,|y_n-y|_{\mathcal{D}({\A})}+|x|_H|({\A}_n-{\A})y|_H.
\end{align}
Therefore, by using H\"older's inequality and also that $\tau \in \mathcal{T}$ is bounded (say, by some $T>0$) and such that there exists $M>0$ for which $\sup_{0 \leq \tau}|X^{x,I}_{t}| \leq M$, we find, 
\begin{align}
\label{eq:domcon-check}
&\E\left[\int_0^{\tau} e^{-\rho t}\Big|\langle X_t^{n;x,I}, \A_{n} D\varphi(X_t^{n;x,I})\rangle_{H} -\langle X_t^{x,I}, {\A} D\varphi(X^{x,I}_t)\rangle_{H}\Big|\,\d t\right] \nonumber \\
&\leq  \left( \E\left[\int_0^{\tau} e^{-\rho t} |X^{n;x,I}_t-X_t^{x,I}|_H^2 \d t\right]\,\E\left[\int_0^{\tau} e^{-\rho t}|D\varphi (X_t^{n;x,I})|^2_{\mathcal{D}({\A})}\d t \right]\right)^{1/2}\\
&
+ M \E\left[\int_0^{\tau} e^{-\rho t} |D\varphi(X_t^{n;x,I})-D\varphi(X_{t}^{x,I})|_{\mathcal{D}({\A})}\d t \right] \nonumber \\
&
+M\E\left[\int_0^{\tau} e^{-\rho t}|({\A}_n-{\A})D\varphi (X^{x,I}_t)|_H\d t \right]. \nonumber
\end{align}
 
We now verify that limits as $n\uparrow \infty$ can be interchanged with the expectations in order to show that the right-hand side of \eqref{eq:domcon-check} converges to zero and thus \eqref{eq:nontrivialconv} holds. 

By Lemma \ref{lem:convdom}(i) and the fact that that $X_{t}^{n;x,I} \to X_{t}^{x,I}$ in probability for every $t\in[0,T]$  when $n \to \infty$, it is not hard to be convinced that the first addend in the right-hand side of \eqref{eq:domcon-check} converges to zero by the dominated convergence theorem. 

Also, recalling \eqref{ClassX} one has for a constant $C>0$
\begin{eqnarray*}
&& \mathds{1}_{[0,\tau]}(t) e^{-\rho t} \Big|D\varphi(X^{n;x,I}_{t}) - D\varphi(X^{x,I}_{t})\Big|^2_{\mathcal{D}(\mathcal{A})} \leq {C}(1+ |X^{n;x,I}_{t}|^2_{H} + |X^{x,I}_{t}|^2_{H}).
\end{eqnarray*}
By Lemma \ref{lem:convdom}(i) the right-hand side of the previous equation are $\P \otimes \d t$-integrable over $\Omega \times [0,T]$, so that Vitali's convergence theorem can be invoked in order to show that the second addend in the right-hand side of \eqref{eq:domcon-check} convergence to zero, upon recalling that $D\varphi \in C(H;\mathcal{D}(\mathcal{A})$ and that, again, $X_{t}^{n;x,I} \to X_{t}^{x,I}$ in probability for every $t\in[0,T]$.

Similarly, also the last addend in the right-hand side of \eqref{eq:domcon-check} can be shown to converge to zero by Vitali's convergence theorem, upon noticing that by the right hand-side of \eqref{convAstar} and \eqref{ClassX} (see also \eqref{eq:graphnorm}) one has, for some $C>0$,
$$|({\A}_n-{\A})D\varphi (X^{x,I}_t)|^2_H \leq 4|{\A}D\varphi (X^{x,I}_t)|^2_H \leq {C}(1 + |X^{x,I}_{t}|^2_{H}),$$
with the last term above being $\P \otimes \d t$-integrable over $\Omega \times [0,T]$ due to Lemma \ref{lem:convdom}(i), and upon using also that $({\A}_n-{\A})y \to 0$ as $n \uparrow \infty$, for any $y \in \mathcal{D}(\mathcal{A})$ (cf.\ \eqref{convAstar}).
\end{proof}

The next proposition provides an inequality of Dynkin type for radial functions. Let us introduce the set of smooth radial functions (centered at $z\in H$) as
\begin{equation}\label{redial}
\mathcal{R}^{z}:=\big\{g: H\to\R: \ g(y):=g_{0}(|y-z|_{H}) \ \mbox{with} \  g_{0}\in C^{2}(\R^{+};\R), \ g'_{0}\geq 0, \  g'_{0}(0)=0\big\}.
\end{equation}
Notice that, if  $g\in \mathcal{R}^{z}$, then $g\in C^{2}(H;\R)$. However, 
since 
\begin{equation}\label{eq:gradient_radial}
D g(y)=\begin{cases}
\displaystyle{\frac{g_0^{\prime}(|y-z|_{H})}{{|y-z|_{H}}} (y-z), \quad \ \ \ \ \mbox{if} \  y \neq z, }\\
0,  \quad\quad  \ \ \ \ \ \ \ \ \ \ \ \ \ \ \  \ \ \ \ \ \ \ \ \ \ \mbox{if} \ y=z,
\end{cases}
\end{equation}  we do not have in general  $Dg(y)\in\mathcal{D}(\A)$, which prevents to use the Dynkin formula provided by  Proposition \ref{Dyn2}. Nevertheless, the inequality 
provided in the new result will allow to include also functions of this type  in the definition of viscosity solutions (see Definition \ref{def:vis}). 
\begin{proposition}
\label{Dyn2} 
Let $x\in H$, $z\in\mathcal{D}(\A)$, $I\in \hat{\mathcal{I}}_0$. Let  $\tau \in \mathcal{T}$ be bounded and such that there exists $M>0$ for which $\sup_{0\leq t \leq \tau}|X^{x,I}_{t}| \leq M$. Let $g_{0}\in C^{2}(\R^{+};\R)$ be nondecreasing and such that $g'_{0}(0)=0$, and  let $g(y):=g_{0}(|y-z|_{H})$.
Then, the following inequality form of Dynkin's formula holds true:
\begin{align*}
\label{Dynkinineq} 
  \E[e^{-\rho \tau}g(X^{x,I}_{\tau})]& \leq 
 g(x)+\E\bigg[\int_{0}^{\tau} e^{-\rho t}\Big[\langle \mathcal{A} z, Dg(X_{t}^{x,I})\rangle_{H} +\frac{1}{2} \mathrm{Tr}[\sigma\sigma^{*} D^{2}g(X_{t}^{x,I})]  -\rho g(X^{x,I}_{t})\Big]\d t\bigg] \\&+\E\bigg[\int_{0^-}^{\tau} e^{-\rho t}\langle Dg(X^{x,I}_{t}),  \vartheta_t\rangle_{H}\d\nu_{t}\bigg].
\end{align*}
\end{proposition}
\begin{proof}
As in the proof of Proposition \ref{Dyn1}, we get
\begin{align*}
&  \E[e^{-\rho \tau}g(X^{n;x,I}_{\tau})]\\&= g(x)+\E\bigg[\int_{0}^{\tau} e^{-\rho t}\Big((\mathcal{G}_{n}-\rho )g\Big)(X^{n;x,I}_{t})\d t\bigg] +\E\bigg[\int_{0^-}^{\tau} e^{-\rho t}\langle Dg(X^{n;x,I}_{t}),  \vartheta_t\rangle_{H}\d\nu_{t}\bigg]\\
& =g(x)+\E\bigg[\int_{0}^{\tau} e^{-\rho t}\Big(\langle \mathcal{A}_{n} X_{t}^{n,x,I}, Dg(X_{t}^{n,x,I})\rangle_{H} +\frac{1}{2} \mbox{Tr}[\sigma\sigma^{*} D^{2}g(X_{t}^{n,x,I})]  -\rho g(X^{n;x,I}_{t})\Big)\d t\bigg] \\& \ \ \ +\E\bigg[\int_{0^-}^{\tau} e^{-\rho t}\langle Dg(X^{n;x,I}_{t}),  \vartheta_t\rangle_{H}\d\nu_{t}\bigg].
\end{align*}
Considering \eqref{eq:gradient_radial}, the fact that $g_{0}'\geq 0$, and  the dissipativity of the Yosida approximants $\A_n$ inherited by the dissipativity of $\mathcal{A}$, we get
\begin{align*}
\langle  \mathcal{A}_{n} X_{t}^{n,x,I}, Dg(X_{t}^{n,x,I})\rangle_{H} &= \langle \mathcal{A}_{n} (X_{t}^{n,x,I}-z), Dg(X_{t}^{n,x,I})\rangle_{H} +  \langle \mathcal{A}_{n} z, Dg(X_{t}^{n,x,I})\rangle_{H}\\& \leq   \langle \mathcal{A}_{n} z, Dg(X_{t}^{n,x,I})\rangle_{H}.
\end{align*}
Hence, 
\begin{align*}
&  \E[e^{-\rho \tau}g(X^{n;x,I}_{\tau})]\\& \leq 
 g(x)+\E\bigg[\int_{0}^{\tau} e^{-\rho t}\Big(\langle \mathcal{A}_{n} z, Dg(X_{t}^{n,x,I})\rangle_{H} +\frac{1}{2} \mbox{Tr}[\sigma\sigma^{*} D^{2}g(X_{t}^{n,x,I})]  -\rho g(X^{n;x,I}_{t})\Big)\d t\bigg] \\&\ \ \ +\E\bigg[\int_{0^-}^{\tau} e^{-\rho t}\langle Dg(X^{n;x,I}_{t}),  \vartheta_t\rangle_{H}\d\nu_{t}\bigg].
\end{align*}
Now, by the reverse Fatou's lemma (thanks to integrability properties as those in the proof of Proposition \ref{Dyn1}) we can take limit superior as $n\to \infty$ in order to obtain the claim.
\end{proof}

The following corollary will be exploited in the proof of the viscosity supersolution property of $V$ (see Theorem \ref{thm:Viscous} below).
\begin{corollary}
\label{Dyn3} 
Let $x\in H$,  $z\in\mathcal{D}(\A)$, $I\in \hat{\mathcal{I}}_0$. Let $\tau \in \mathcal{T}$ be bounded and assume that, for some $\varepsilon>0$, it holds
$
|X^{x,I}_{t}-z|_{H}\leq \varepsilon$ for all  $t\in[0,\tau].
$
Then, there exists a constant $c_{o}>0$ such that 
\begin{equation}
\begin{aligned}
\label{Dynkinineq} 
&\E\left[e^{-\rho\tau}|X^{x,I}_{\tau}-z|^3_{H}\right]\leq c_{o} (\varepsilon+\varepsilon^2)\,\E		\left[\int_{0^-}^{\tau}e^{-\rho t} (\d t+\d \nu_{t})\right].
\end{aligned}
\end{equation}
\end{corollary}
\begin{proof}
Notice that, setting
$$g(y):=|y-z|_H^3, \quad y \in H,
$$
one has
$$
Dg(y)=3|y-z|_H (y-z), \ \ \ D^2g(y)=3|y-z|_{H}^{-1}(y-z)\otimes (y-z) +3|y-z|_H \cdot \textbf{Id}_H.
$$
In particular
$$
\mbox{Tr}\left[\sigma\sigma^{*} D^{2}g(y)\right]\leq |D^{2}g(y)|_{\mathcal{L}(H)} \mbox{Tr}\left[\sigma\sigma^{*}\right]\leq 6 |y-z|_{H}|\sigma\sigma^{*}|_{\mathcal{L}_{1}(H)}.
$$
Then,  the claim follows from Proposition \ref{Dyn2}.
\end{proof}

\subsubsection{Dynamic Programming Principle}

The next result is the dynamic programming principle for the singular stochastic control problem \eqref{eq:newV}. We provide a sketch of its proof in Appendix \ref{sec:appendixA}.

\begin{proposition}[Dynamic Programming Principle for $V$]
\label{DPPOC}
Recall \eqref{eq:admissiblesetfinal}. We have 
\begin{equation}
\label{eq:DP}
V(x)=\inf_{(I, \tau^I)  \in\hat{\mathcal{I}_0} \times \mathcal{T}}\E\left[\int_{0^-}^{\tau^I} e^{-\rho t} \left(G(X^{x,I}_{t})\d t+\d \nu_{t}\right)+ e^{-\rho \tau^I} V(X_{\tau^I}^{x,I})\right], \quad x \in H.
\end{equation}
\end{proposition}


\subsubsection{The value function as viscosity solution of the variational inequality}
We are  ready to provide the definition of viscosity solution to \eqref{VIOC} and to prove that $V$ defined in \eqref{eq:newV} is actually a viscosity solution to the variational inequality.
\begin{definition}[Viscosity solution]\label{def:vis}
\medskip
\begin{enumerate}[(i)]
\item[]
\medskip
\item 
\medskip
We say that $v\in C(H)$ is a viscosity supersolution to \eqref{VIOC} at $x\in H$ if, for every $\psi=\varphi +g$ with  $\varphi\in \mathcal{X}$, $g\in\mathcal{R}^{0}$  such that $0=v(x)-\psi (x)=\min(v-\psi)$, one has
$$\max\Big\{\rho\psi(x) -\frac{1}{2}\mathrm{Tr}[\sigma\sigma^{*} D^{2}\psi (x)]-\langle x,\A D\varphi(x)\rangle_{H}-G(x), \ \sup_{\theta\in\Theta}\big\{- \langle D\psi(x),\theta\rangle_{H}-1\big\} \Big\}\geq 0.$$
\item We say that $v\in C(H)$ is a viscosity subsolution to \eqref{VIOC} at $x\in H$ if,  for every $\psi=\varphi -g$ with  $\varphi\in \mathcal{X}$, $g\in\mathcal{R}^{0}$  such that $0=v(x)-\psi (x)=\max(v-\psi)$, one has
$$\max\Big\{\rho\psi(x) -\frac{1}{2}\mathrm{Tr}[\sigma\sigma^{*} D^{2}\psi (x)]-\langle x,\A D\varphi(x)\rangle_{H}-G(x), \ \sup_{\theta\in\Theta}\big\{- \langle D\psi(x),\theta\rangle_{H}-1\big\} \Big\}\leq 0.$$
\item We say that $v\in C(H)$ is a viscosity solution to \eqref{VIOC} at $x\in H$ if it is both a viscosity super- and subsolution.
\end{enumerate}
\end{definition}
\begin{remark}\label{rem:vis}
In the context of classical (regular) stochastic control, the definition of viscosity solution based on the  set of test functions $\mathcal{X}$ combined with the set of radial functions $\mathcal{R}^{0}$ is typically sufficient to prove comparison results for viscosity solutions of the PDE (see \cite[Ch.\,3]{FGS}); hence,  to characterize the value function of the control problem as the unique solution to the associated dynamic programming  equation. 

In this paper, we do not address the relevant topic of uniqueness (see  Remark \ref{rem:uniqueness} below), leaving it as a future research topic.  Instead, we will use test functions from the set $\mathcal{X}$  of quadratic form  in Section \ref{sec:connection} in order to prove our regularity result (see, in particular, the proof of Proposition \ref{prop:regularity}).
\end{remark}

\begin{theorem}
\label{thm:Viscous}
$V$ is a viscosity solution to \eqref{VIOC} at all $x\in H$.
\end{theorem}
\begin{proof}
\medskip
(\emph{Subsolution property.})
Let  $x\in H$ and $\psi=\varphi+g$ with $\varphi\in \mathcal{X}$, $g\in \mathcal{R}^{0}$  such that $0=V(x)-\psi(x)=\max(V-\psi)$.
\medskip

\emph{Step 1.}
For $\theta \in \Theta$ and $\zeta>0$, consider the control
$$I_\cdot=(\vartheta_\cdot,\nu_\cdot)\equiv (\theta,\hat{\nu})\in \hat{\mathcal{I}_0},$$
with $\hat{\nu}_{0^-}=0$ and $\hat{\nu}_t=\zeta$ for any $t\geq0$.
By Proposition \ref{DPPOC}, we have for $h>0$
\begin{equation}
\label{eq:conseqDPP}
V(x)\leq\E\left[\int_{0}^{h} e^{-\rho t} G(X^{x,I}_{t})\d t + \zeta + e^{-\rho h} V(X_{h}^{x,I})\right].
\end{equation}
We now aim at taking limits as $h\to 0^+$ in \eqref{eq:conseqDPP}. To that end, note that the limits in the right-hand side of \eqref{eq:conseqDPP} can be interchanged with the expectation: By the monotone convergence theorem for the integral term; because of the dominated convergence theorem for the third addend, since $V$ has sub-quadratic growth (cf.\ Proposition \ref{prop:regularityV}), and since $X^{x,I}_h= X^{x,0}_h + \zeta \theta$ (cf.\ \eqref{eq:state}) and $\E[\sup_{t \in [0,T]}|X^{x,0}_t|^2] < \infty$, for any $T>0$, by \eqref{estimatesup}. Hence, due to $X^{x,I}_h\to x+ \zeta \theta$ as $h\to 0^+$, we find
\begin{equation}
\label{eq:limitsub}
V (x)\leq \zeta + V(x+\zeta \theta). 
\end{equation}

But then, \eqref{eq:limitsub} yields
$$
\psi (x)\leq \zeta + \psi(x+\zeta \theta),
$$
which, dividing by $\zeta$ and letting $\zeta\to 0^+$, in turn gives (cf.\ Proposition \ref{prop:regularityV})
$$-\langle D\psi(x), \theta\rangle_H\leq 1.$$
Since the latter holds for every $\theta\in\Theta$, we find
$$\sup_{\theta\in\Theta} \left\{-\langle D\psi(x),\theta\rangle_H-1\right\}\leq 0.$$
\medskip

\emph{Step 2.}
Let now $I=\textbf{0}$ be the null control. Setting
$$\tau_R:=\inf\{t\geq0:\, |X^{x,0}_t|_H \geq R\}$$ (with $\inf \emptyset = + \infty$) and letting $h>0$,  by  Proposition \ref{DPPOC} we have
$$
\psi(x)\leq\E\left[\int_{0}^{\tau_R \wedge h} e^{-\rho t} G(X^{x,0}_{t})\d t+ e^{-\rho(\tau_R \wedge h)} \psi(X_{\tau_R \wedge h}^{x,0})\right].
$$
Using now Propositions \ref{Dyn1} and \ref{Dyn2}, we find
$$
\E\bigg[\int_{0}^{\tau_R \wedge h} e^{-\rho t} \Big(\rho\psi(X_t^{x,0}) -\frac{1}{2}\mathrm{Tr} [\sigma\sigma^{*}D^{2}\psi(X_t^{x,0})]- \langle X_t^{x,0}, \mathcal{A}D\varphi(X_t^{x,0})\rangle - G(X^{x,0}_{t})\Big)\d t\bigg]\leq 0.
$$
Dividing by $h$, recalling that $\varphi \in \mathcal{X}$ and using that $\E[\sup_{t \in [0,T]}|X^{x,0}_t|] < \infty$, for any $T>0$ by \eqref{estimatesup}, we can invoke the integral mean-value and the dominated convergence theorems when letting $h\to 0^+$, and we obtain 
$$\rho\psi(x) -\frac{1}{2}\mathrm{Tr}[\sigma\sigma^{*} D^{2}\psi (x)]-\langle x,\A D\varphi(x)\rangle_{H}-G(x)\leq 0$$.
\medskip

\emph{Step 3.} Combining the last two steps we obtain the desired subsolution property of $V$.

\bigskip
(\emph{Supersolution property.})
Let now $x\in H$, $\psi = \varphi-g$, with  $\varphi\in \mathcal{X}$ and $g\in\mathcal{R}^{0}$ be such that $0=V(x)-\psi(x)=\min(V-\psi)$. Assume, by contradiction that there exists $\eta>0$ such that 
\begin{equation}\label{C1}
\sup_{\theta\in\Theta} \left\{-\langle D\psi(x),\theta\rangle_H-1\right\}\leq -2\eta
\end{equation}
and 
\begin{equation}\label{C2}
\rho \psi (x)-\frac{1}{2}\mathrm{Tr}[\sigma\sigma^{*} D^{2}\psi(x)] -\langle x,\A D\varphi (x)\rangle_{H}-G(x)\leq -2\eta.
\end{equation}
By continuity, for a suitable $\varepsilon_{o}>0$, 
\begin{equation}\label{C3}
\sup_{\theta\in\Theta} \left\{-\langle D\psi(y),\theta\rangle_H-1\right\}\leq -\eta, \  \ \ \ \forall y\in B_{|\cdot|_{H}}(x,\varepsilon_{o}).
\end{equation}
and 
\begin{equation}\label{C4}
\rho \psi (y)-\frac{1}{2}\mathrm{Tr}[\sigma\sigma^{*} D^{2}\psi(y)] -\langle y,\A D\varphi (y)\rangle_{H}-G(y)\leq -\eta, \ \ \ \  \forall y\in B_{|\cdot|_{H}}(x,\varepsilon_{o}),
\end{equation}
where $B_{|\cdot|_{H}}(x,\varepsilon_{o}):=\{y \in H: |y-x|_H \leq \varepsilon_{o}\}.$

Let now $I=(\vartheta,\nu)\in\hat{\mathcal{I}}_0$ be  arbitrary but fixed, and set, for $\varepsilon\in (0,\varepsilon_{o})$,
$$\tau_{\varepsilon}^I:=\inf\{t\geq 0: \ X_t^{x,I} \notin B_{|\cdot|_{H}}(x,\varepsilon)\},$$
with the convention $\inf\emptyset=\infty$, which still provides a sense to the formulae below. In the following, we are going simply to write $\tau_{\varepsilon}$ instead of $\tau_{\varepsilon}^I$, unless it becomes important to stress the explicit dependence on $I$. 

Now pick $z\in \mathcal{D}(\A)$ such that $|z-x|<\varepsilon/2$, which is possible as $\mathcal{D}(\A)$ is dense in $H$.
Employing \eqref{C3}, \eqref{C4}, and Propositions \ref{Dyn1}-\ref{Dyn2} and Corollary \ref{Dyn3}, we find for each $T>0$
\begin{align*}
&    V(x)-\E[e^{-\rho{(\tau_{\varepsilon}\wedge T)}} V(X_{{\tau_{\varepsilon}\wedge T}}^{x,I})]\leq {\psi}(x)-\E[e^{-\rho {(\tau_{\varepsilon}\wedge T)}}{\psi}(X_{\tau_{\varepsilon}}^{x,I})]
\\  &= - \E\left[e^{-\rho {(\tau_{\varepsilon}\wedge T)}}\big|X_{{\tau_{\varepsilon}\wedge T}}^{x,I}-z\big|^{3}_{H}\right]+  \E\left[e^{-\rho {(\tau_{\varepsilon}\wedge T)}}\big|X_{{\tau_{\varepsilon}\wedge T}}^{x,I}-z\big|^{3}_{H}\right]+ {\psi}(x)-\E[e^{-\rho {(\tau_{\varepsilon}\wedge T)}}{\psi}(X_{\tau_{\varepsilon}}^{x,I})]\\
   &\leq -\frac{\varepsilon^{3}}{8}\E\left[e^{-\rho{(\tau_{\varepsilon}\wedge T})}\mathds{1}_{{{\{\tau_{\varepsilon}< T\}}}}\right]
   +\E\left[c_{o} (\varepsilon+\varepsilon^2)\int_{0^-}^{\tau}e^{-\rho t} (\d t+\d \nu_{t})\right]
\\
&+ \E\left[\int_0^{{\tau_{\varepsilon}\wedge T}}e^{-\rho t}\Big(\rho\psi(X_t^{x,I})- \frac{1}{2}\mathrm{Tr}[\sigma\sigma^{*}D^{2}\psi(X^{x,I}_{t})]-\langle X^{x,I}_{t}, \A D\varphi(X^{x,I}_{t})\rangle_{H}\Big)\d t\right] \\ 
& - \E\left[\int_{0^-}^{\tau_{\varepsilon}\wedge T}e^{-\rho t}\langle D\psi(X^{x,I}_{t}), \vartheta_{t}\rangle_{H}\d \nu_{t}\right]\\
&\leq  -\frac{\varepsilon^{3}}{8}\E\left[e^{-\rho{(\tau_{\varepsilon}\wedge T})}\mathds{1}_{{{\{\tau_{\varepsilon}< T\}}}}\right]+\E\left[c_{o} (\varepsilon+\varepsilon^2)\int_{0^-}^{\tau}e^{-\rho t} (\d t+\d \nu_{t})\right]\\
&
+  \E\left[\int_{0^{-}}^{{\tau_{\varepsilon}\wedge T}}e^{-\rho t}(G(X_t^{x,I})\d t+\d\nu_{t})\right]-\eta  \E\left[\int_{0^{-}}^{{\tau_{\varepsilon}\wedge T}}e^{-\rho t}(\d t+\d\nu_{t})\right].
\end{align*}
Letting $T\uparrow \infty$ in the above inequality and using the fact that $e^{-\rho\tau_{\varepsilon}}\mathds{1}_{\{\tau_{\varepsilon} = + \infty\}} = 0$ by \eqref{eq:convention}, we obtain
\begin{align}
\label{eq:superV}
&    V(x)-\E\left[\int_{0^{-}}^{\tau_{\varepsilon}}e^{-\rho t}(G(X_t^{x,I})\d t+\d\nu_{t})+e^{-\rho\tau_{\varepsilon}} V(X_{\tau_{\varepsilon}}^{x,I})\right] \\&
\leq - \frac{\varepsilon^{3}}{8}\E\left[e^{-\rho\tau_{\varepsilon}}\right] \nonumber 
 +\left(c_{o} (\varepsilon+\varepsilon^2)-\eta\right)  \E\left[\int_{0^{-}}^{\tau_{\varepsilon}}e^{-\rho t}(\d t+\d\nu_{t})\right],
\end{align}
The above inequality was obtained for a 
 fixed  $I\in \hat{\mathcal{I}}_{0}$.
 Now, stressing again the dependency of $\tau_{\varepsilon}$ with respect to $I\in \hat{\mathcal{I}}_{0}$, we take the supremum on both terms of \eqref{eq:superV} with respect to $I\in\hat{\mathcal{I}}_{0}$, we use \eqref{eq:DP}, and consider the identity
$$e^{-\rho\tau_{\varepsilon}^I}+\rho \int_0^{\tau_{\varepsilon}^I}e^{-\rho t}\d t= 1,$$
to obtain
\begin{align*}
&   0   \leq \sup_{I\in\hat{\mathcal{I}}_{0}}\left\{-\frac{\varepsilon^{3}}{8}\E\left[{1-  \rho} \int_0^{\tau_{\varepsilon}^I}e^{-\rho t}\d t\right]
+ \left(c_{o} (\varepsilon+\varepsilon^2)-\eta\right)\E\left[\int_{0^-}^{\tau_{\varepsilon}^I} e^{-\rho t}(\d t+ \d \nu_{t})\right]\right\};
\end{align*}
that is,
\begin{align*}
&   0   \leq - \frac{\varepsilon^{3}}{8}+\sup_{I\in\hat{\mathcal{I}}_0}
\E\left[ \left(c_{o} (\varepsilon+\varepsilon^2)+\frac{\rho\varepsilon^{3}}{8}-\eta\right) \int_{0^-}^{\tau_{\varepsilon}^I} e^{-\rho t}(\d t+ \d \nu_{t})\right],
\end{align*}
Taking $\varepsilon$ small enough with respect to $\eta$, which can be done without loss of generality,  the supremum is nonpositive. Hence, we reach a contradiction and the supersolution property of $V$ is proved.
\end{proof}
\color{black}
\begin{remark}
\label{rem:proofvisc}
It is worth emphasizing that the argument of the proof of the supersolution property can be clearly applied also in the finite-dimensional setting. As a consequence, this novel argument is able to reduce substantially the technicalities that are typically needed in $\mathbb{R}^n$, $n\geq1$, in order to show the supersolution (respectively, subsolution) property in minimization problems (respectively, maximization problems) involving singular controls (see, e.g., \cite{Ch}, \cite[Ch.\ VIII]{FlemingSoner}, \cite{HS}, and \cite{Ma}, among others).
\end{remark}

\begin{remark}\label{rem:uniqueness}
To the best of our knowledge, there are only two contributions in the literature in which the viscosity approach is developed for variational inequalities in infinite dimensional spaces: These are \cite{CPP} and \cite{GS}. In both of them, the variational inequality is related to an obstacle problem and thus involves a constraint on the solution itself, and not on its gradient (as, instead, it is in our case). 
\begin{enumerate}[(i)]
\item In \cite{CPP}, a problem of optimal stopping for a stochastic process with memory is considered. The  infinite-dimensional space in which the underlying state process takes values is that of continuous functions, and the evolution of the process is subject to the action of an unbounded operator (the first-order derivative). However, the Brownian noise is finite-dimensional in \cite{CPP}.
\item In  \cite{GS}, the aim is to price American options defined on forward-rates models. The interest rate process takes values in an Hilbert space and it evolves according to nonlinear dynamics, which however do not involve an unbounded operator. In this framework, the authors are able to prove a comparison result and to apply it in order to (uniquely) characterize the value function of the underlying infinite-dimensional optimal stopping problem. Remark \ref{rem:uniqueness2} below will articulate more on the relation between our paper and \cite{GS}.
\end{enumerate}
\end{remark}


\section{Selecting a specific direction of action: A related optimal stopping problem}
\label{sec:connection}

In order to achieve further regularity of $V$, we now specialize to the case in which the controller can only choose the intensity $\nu \in \mathcal{S}$ appearing in the decomposition $(\vartheta,\nu)\in \mathcal{I}_0$ of any admissible control $I \in \mathcal{I}$. 

In particular, we select the convex cone 
$$\Delta = \Span\{\hat{n}\},$$ 
where $\hat{n} \in H_+ \setminus \{\textbf{0}\}$ and it is such that, without loss of generality, $\langle q, \hat{n}\rangle_H =1$. Consequently (cf.\ \eqref{setTHETA}), $\Theta=\{\hat{n}\}$
and we make the following assumption, which will hold throughout the rest of the paper without further mention.

\begin{assumption}
 \label{ass:n}   
The vector $\hat{n} \in H_+ \setminus \{\textbf{0}\}$ is an eigenvector of $\A$ with eigenvalue $\lambda$\footnote{Clearly, with $\lambda\leq -\delta < 0$, where $\delta$ is  as in Assumption \ref{ass:A}.}.
\end{assumption}

\begin{remark}
\label{rem:eigenvector}
\begin{itemize}
\item[(i)] It is worth stressing that Assumption \ref{ass:n} does not imply that there exists a singularly controlled component of the state variable $X$ that is decoupled from the rest of the infinite-dimensional dynamics of $X$. To clarify this, suppose that $H=\mathbb{R}^2$ and consider
\begin{equation*}
\mathcal{A} = \left(
\begin{array}{cc}
1 & 0 \\
1 & 1
\end{array}
\right)
\quad \text{and} \quad 
\hat{n}=\left(
\begin{array}{c}
0 \\ 
1
\end{array}
\right).
\end{equation*}
\noindent Notice $\hat{n}$ is an eigenvector of $\mathcal{A}$. In this case, the two components of the state process $X$ are truly coupled through the operator $\mathcal{A}$ so that the action along the direction $\hat{n}$ affects indirectly also the first component of $X$.
\item[(ii)] Sufficient conditions guaranteeing the validity of Assumption \ref{ass:n} can be found in \cite{Arendt-etal, Clement-etal}. We refer to Section 4 in \cite{Calvia-etal} for a detailed discussion.
\end{itemize}
\end{remark}

\begin{remark}
\label{rem:1dSSC}
Within this section, the decision maker can act only on the intensity of the control, which is modeled as a one-dimensional singular control process belonging to the set $\mathcal{S}$. This assumption is made in order to develop regularity results for the value function $V$ by linking it to an optimal stopping problem. Indeed, the connection between singular stochastic control and optimal stopping is well understood for one-dimensional control processes of bounded variation, whereas it has not yet been established for multi-dimensional controls. Moreover, when the dimensionality of the problem is greater than one (but still finite) and the control is multi-dimensional, analyzing the geometry of the state space and constructing an optimal control rule of reflecting type already present significant challenges; see the introduction of \cite{DianettiFerrari} for a discussion of this point. These challenges become even more substantial in infinite-dimensional settings, where it is not even clear what precise meaning should be assigned to a “reflecting policy.” These highly intriguing aspects remain important open questions for future research.
\end{remark}

For future use, we recall here that the uncontrolled state-process uniquely solves in the mild sense (cf.\ \eqref{eq:state})
$$\d X^{x,0}_{t}=\mathcal{A}X^{x,0}_{t}\d t+\sigma \d W_{t},  \ \ \ X_{0}^{x,0}=x \in H;
$$
that is (cf.\ \eqref{eq:stochintegr}),
\begin{equation}
\label{eq:mild-uncontr}
X_{t}^{x,0}=e^{t\A}x + \int_{0}^t{e}^{(t-s)\A}\sigma \d W_{s} = e^{t\A}x + W^{\mathcal{A},\sigma}_t, \quad t\geq 0.
\end{equation}
Furthermore, $X^{x,0}$ has continuous sample paths.

For our subsequent analysis,  we define
\begin{equation}
\label{eq:hatG}
G_{\hat{n}}(x):=\langle DG(x), \hat{n} \rangle_H, \ \ \ x\in H,
\end{equation}
and introduce the optimal stopping problem
$$\sup_{\tau \in \mathcal{T}} \E\bigg[\int_0^{\tau} e^{-(\rho-\lambda)t} G_{\hat{n}}(X^{x,0}_t) \d t - e^{-(\rho-\lambda)\tau}\bigg], \quad x \in H,$$
where $\mathcal{T}$ denotes the set of the $\mathbb{F}$-stopping times.
The next result relates the previous optimal stopping problem to the directional derivative of $V$ as in \eqref{eq:newV}, in the direction $\hat{n}$. This can be thought of as an infinite-dimensional analogue of the result in \cite{BK}, where, in a finite-dimensional setting, it is proven that a suitable optimal timing problem provides the marginal value of an irreversible investment problem of monotone follower type.

\begin{theorem}
\label{thm:connection}
For any $x \in H$, one has
\begin{equation}
\label{eq:v}
 V_{\hat{n}}(x):=\langle DV(x), \hat{n} \rangle_H = \sup_{\tau \in \mathcal{T}} \E\bigg[\int_0^{\tau} e^{-(\rho-\lambda)t} G_{\hat{n}}(X^{x,0}_t) \d t - e^{-(\rho-\lambda)\tau}\bigg].    
\end{equation}
\end{theorem}

\begin{proof}
The proof is organized in two steps.
\vspace{0.25cm}

\emph{Step 1.} Let $x\in H$. Here we prove that there exists an optimal control $I^{\star}:=(\hat{n},{\nu}^{\star}) \in \mathcal{I}_0$ for $V(x)$ as in \eqref{eq:newV}. Furthermore, if $G$ is strictly convex, $I^{\star}$ is unique up to indistinguishability. 

Let $(I^k)_{k\in \mathbb{N}} :=(\hat{n}, \nu^k)_{k\in \mathbb{N}} \subseteq \mathcal{I}_0$ be a minimizing sequence for $V(x)$. Let us denote by $X^{x,k}:=X^{x,I^k}$, $k\in \mathbb{N}$. Given that $V(x) \leq \hat{c}_o(1+|x|^{2}_H)$, $x\in H$, (cf.\ Proposition \ref{prop:regularityV}), we have for some $\varepsilon>0$
\begin{equation}
\label{eq:exOCstima1}
\sup_{k\in\mathbb{N}}\E\bigg[\int_{0^-}^{\infty}e^{-\rho t} |X^{x,k}_t|_H^{2}\, \d t\bigg] \leq \frac{1}{\kappa_1}\Big(\hat{c}_o(1+|x|^{2}_H) + \varepsilon + \frac{\kappa_2}{\rho}\Big).
\end{equation}

Because
$$X^{x,k}_t = e^{t\A}x + W^{\A, \sigma}_t + \widehat{n}\int_{0^-}^t e^{\lambda(t-s)}\d \nu^k_t,$$
simple estimates using $\lambda<0$, \eqref{estimatesup-W}, and $|e^{t\mathcal{A}}|_{\mathcal{L}(H)} \leq e^{-\delta t}$, for all $t\geq0$, give, for a constant $C>0$ (changing from line to line),

\begin{equation}
    \label{eq:existence-stima1}
    |\widehat{n}|_H^2 e^{2\lambda t} \E\big[|\nu^k_t|^2\big] \leq C\Big(|x|_H^2 + \overline{c}_2 + \E\big[|X^{x,k}_t|^2\big]\Big).
\end{equation}
This in turn yields by \eqref{eq:exOCstima1}
\begin{equation}
    \label{eq:existence-stima2}
\sup_{k\in \mathbb{N}}\E\bigg[\int_0^{\infty }e^{-(\rho - 2\lambda)t} |\nu^k_t|^2 \d t\bigg] \leq  C\Big( 1+ |x|_H^2 +  \sup_{k\in \mathbb{N}} \E\bigg[\int_0^{\infty} e^{-\rho t}|X^{x,k}_t|^2 \d t\bigg]\Big) < \infty.
\end{equation}

Hence, by Banach-Saks theorem, there exist a subsequence of $(\nu^{k})_{k\in \mathbb{N}}$ - still denoted by $(\nu^{k})_{k\in \mathbb{N}}$ - and some ${\nu}^{\star} \in L^2(\Omega \times[0,\infty); \P \otimes e^{-(\rho-2\lambda)t} \d t)$ such that
\begin{equation}
\label{nuhat}
\widehat{\nu}^j:=\frac{1}{j + 1}\sum_{k=0}^j \nu^k \to {\nu}^{\star} \quad \text{in} \,\, L^2(\Omega \times[0,\infty); \P \otimes e^{-(\rho-2\lambda)t} \d t).
\end{equation}
Actually, up to passing to a further subsequence (again, relabeled in the following), such a convergence can be realized $\P \otimes e^{-(\rho-2\lambda)t} \d t$-a.e. Then, by arguing as in \cite[Lemmata 4.5-4.7]{KaratzasShreve84}, the process $\nu^{\star}$ admits a modification - still denoted by $\nu^{\star}$ - that is right-continuous, nondecreasing and $\mathbb{F}$-adapted, and thus belongs to $\mathcal{S}$.
Furthermore, given that ${\nu}^{\star} \in L^2(\Omega \times[0,\infty); \P \otimes e^{-(\rho-2\lambda)t} \d t)$ and that $t \mapsto \E[|\nu^{\star}_t|^2]$ is nondecreasing, it follows that the integrability condition required in \eqref{setI0} is also met (recall that, in this section, $\vartheta_t\equiv \widehat{n}$).
Set then $\widehat{I}^{j}:=(\hat{n},\widehat{\nu}^{j}) \in \mathcal{I}_0$, $j\in \mathbb{N}$, and  $I^{\star}:=(\hat{n},\nu^{\star}) \in \mathcal{I}_0$.

Notice now that by making use of an integration by parts in the integrals with respect to $\d \widehat{\nu}^j$, one also has that $\P \otimes e^{-\rho t}\d t$-a.e.\ (cf.\ \eqref{eq:stochintegr} and \eqref{mild})
$$X_{t}^{x,\widehat{I}^{j}}=e^{t\A}x+W_{t}^{\A,\sigma}+ \int_{0^-}^{t}e^{(t-s)\A} \hat{n}\, \d \widehat{\nu}^j_{s} = e^{t\A}x+W_{t}^{\A,\sigma}+ \hat{n} \int_{0^-}^{t} e^{\lambda(t-s)} \d \widehat{\nu}^j_{s} \to X_{t}^{x,I^\star},$$
and also $\P$-a.s.
$$\int_{0^-}^{\infty}e^{-\rho t} \d \widehat{\nu}^j_t \to \int_{0^-}^{\infty}e^{-\rho t} \d \nu^{\star}_t.$$

By convexity of $(x,\nu) \mapsto \mathcal{J}(x;I)$ (cf.\ \eqref{eq:costfunctbis}) the sequence $(\widehat{I}^j)_{j\in \mathbb{N}}:=(\hat{n},\widehat{\nu}^j)_{j\in \mathbb{N}}$ is minimizing as well, and we conclude by Fatou's lemma and the previous convergences that
$$V(x)=\liminf_{j\to\infty} \mathcal{J}(x;\widehat{I}^j) \geq \mathcal{J}(x;I^{\star}), \quad x \in H,$$
from which the optimality of $I^{\star}=(\hat{n},\nu^{\star}) \in \mathcal{I}_0$ for $V(x)$ follows.

Finally, the uniqueness claim follows from strict convexity of $G$, upon using arguments as those in the proof of Proposition 3.4 in \cite{FedericoPham}.
\vspace{0.25cm}

\emph{Step 2.} Here we prove \eqref{eq:v}. Since the proof very much follows the lines of those of Lemmata 3 and 4 in \cite{BK}, we only sketch it. 

Let $x \in H$, $I^{\star}=(\hat{n},\nu^{\star}) \in \mathcal{I}_0$ be optimal for $V(x)$ (cf.\ \emph{Step 1} above), and for $\tau \in \mathcal{T}$ and $\varepsilon>0$, define
\begin{equation}
\label{def:xi}
\xi_t:= \begin{cases} 
      I^{\star}_t, & 0 \leq t < \tau, \\
      I^{\star}_t + \varepsilon e^{\lambda \tau}\hat{n}, & t \geq \tau.
   \end{cases}
	\end{equation}
The process $\xi \in \mathcal{I}_0$ and it has direction of action $\hat{n}$ and intensity of action $\nu^{\xi}$ such that $\nu^{\xi}_{0^-}=0$ and $\d\nu^{\xi}_t = \d\nu^{\star}_t + \varepsilon e^{\lambda t} \delta(t-\tau)$, for any $t\geq0$.
Furthermore, letting $X^{x-\varepsilon \hat{n}, \xi}$ be the solution to \eqref{eq:state} started at time $0^-$ from level $x-\varepsilon \hat{n}$ and controlled by $\xi$, we have that: $X^{x-\varepsilon \hat{n}, \xi} \equiv X^{x-\varepsilon \hat{n}, I^{\star}}$ on $[0,\tau)$, while $X^{x-\varepsilon \hat{n}, \xi} \equiv X^{x, I^{\star}}$ on $[\tau,\infty)$.

Hence, by exploiting the convexity of $G$, Assumption \ref{ass:A2}, and the previous observations,
\begin{align}
\label{DV-ineq1}
& V(x) - V(x-\varepsilon \hat{n}) \geq \mathcal{J}(x;{I}^{\star}) - \mathcal{J}(x-\varepsilon \hat{n};\xi)\nonumber \\
& = \E\bigg[\int_0^{\tau}e^{-\rho t} \langle DG(X^{x-\varepsilon \hat{n}, I^{\star}}_t), X^{x, I^{\star}}_t - X^{x -\varepsilon \hat{n}, I^{\star}}_t \rangle_H \d t -\varepsilon e^{-(\rho-\lambda)\tau}\bigg] \nonumber \\
& = \varepsilon \E\bigg[\int_0^{\tau}e^{-(\rho-\lambda) t} \langle DG(X^{x-\varepsilon \hat{n}, I^{\star}}_t), \hat{n} \rangle_H \d t -e^{-(\rho-\lambda)\tau}\bigg] \\
& \geq \varepsilon \E\bigg[\int_0^{\tau}e^{-(\rho-\lambda) t} \langle DG(X^{x-\varepsilon \hat{n}, 0}_t), \hat{n} \rangle_H \d t -e^{-(\rho-\lambda)\tau}\bigg], \nonumber
\end{align}
from which
\begin{equation}
\label{DV-ineq2}
\liminf_{\varepsilon \to 0} \frac{1}{\varepsilon} \Big( V(x) - V(x-\varepsilon \hat{n}) \Big) \geq \sup_{\tau \in \mathcal{T}} \E\bigg[\int_0^{\tau} e^{-(\rho-\lambda)t} G_{\hat{n}}(X^{x,0}_t) \d t - e^{-(\rho-\lambda)\tau}\bigg].
\end{equation}

Under the usual convention $\inf\emptyset=+\infty$, let now $\tau^{\star}:=\inf\{t\geq 0:\, \nu^{\star}_t>0\}$ and $\tau^{\varepsilon}:=\inf\{t\geq 0:\, \nu^{\star}_t \geq \varepsilon e^{\lambda t}\}$, for $\varepsilon>0$. Clearly, $\tau^{\varepsilon} \downarrow \tau^{\star}$ as $\varepsilon \downarrow 0$.
We then define
\begin{equation}
\label{def:eta}
\eta_t:= \begin{cases} 
      0, & 0 \leq t < \tau^{\varepsilon}, \\
      I^{\star}_t - \varepsilon e^{\lambda \tau^{\varepsilon}}\hat{n}, & t \geq \tau^{\varepsilon},
   \end{cases}
	\end{equation}
and notice that denoting by $X^{x + \varepsilon \hat{n}, \eta}$ the solution to \eqref{eq:state} started at time $0^-$ from level $x + \varepsilon \hat{n}$ and controlled by $\eta$, we have that: $X^{x+\varepsilon \hat{n}, \eta} \equiv X^{x+\varepsilon \hat{n}, 0}$ on $[0,\tau^{\varepsilon})$, while $X^{x+\varepsilon \hat{n}, \eta} \equiv X^{x, I^{\star}}$ on $[\tau^{\varepsilon},\infty)$.

Convexity of $G$ then yields
\begin{align}
\label{DV-ineq3}
& V(x+\varepsilon \hat{n}) - V(x) \leq \mathcal{J}(x+\varepsilon \hat{n};\eta) - \mathcal{J}(x;I^{\star}) \\
& = \E\bigg[\int_0^{\tau^{\varepsilon}}e^{-\rho t} \langle DG(X^{x+\varepsilon \hat{n}, 0}_t), X^{x+\varepsilon \hat{n}, 0}_t - X^{x, I^{\star}}_t \rangle_H \d t - \int_{0^-}^{\tau^{\varepsilon}} e^{-\rho t} \d\nu^{\star}_t - \Big(\varepsilon e^{-\lambda\tau^{\varepsilon}} - \nu^{\star}_{\tau^{\varepsilon}}\Big) e^{-\rho\tau^{\varepsilon}}\bigg]. \nonumber
\end{align}
Notice now that
$$X^{x+\varepsilon \hat{n}, 0}_t - X^{x, I^{\star}}_t = \hat{n} e^{\lambda t} \Big(\varepsilon - \int_{0^-}^t e^{-\lambda s}\d \nu^{\star}_s \Big).$$
Plugging the latter relation into \eqref{DV-ineq3}, dividing by $\varepsilon$ and adding and substracting terms, one arrives at
\begin{align}
\label{DV-ineq4}
& \frac{1}{\varepsilon}\Big( V(x + \varepsilon \hat{n}) - V(x)\Big) \leq \E\bigg[\int_0^{\tau^{\star}}e^{-(\rho-\lambda) t} \langle DG(X^{x, 0}_t), \hat{n} \rangle_H \d t - e^{-(\rho-\lambda)\tau^{\star}}\Bigg] \nonumber \\
& + \E\bigg[\int_0^{\tau^{\star}}e^{-(\rho-\lambda) t} \langle DG(X^{x+\varepsilon \hat{n}, 0}_t) - DG(X^{x, 0}_t), \hat{n} \rangle_H \d t\bigg] + \E\Big[e^{-(\rho-\lambda)\tau^{\star}} - e^{-(\rho-\lambda)\tau^{\varepsilon}}\Bigg] \\
& + \E\bigg[\int_{\tau^{\star}}^{\tau^{\varepsilon}}e^{-(\rho-\lambda) t} \langle DG(X^{x+\varepsilon \hat{n}, 0}_t), \hat{n} \rangle_H \Big(1 - \frac{1}{\varepsilon} \int_{0^-}^t e^{-\lambda s} \d\nu^{\star}_s\Big) \d t\bigg]. \nonumber
\end{align}
Taking limits as $\varepsilon \downarrow 0$ it is not difficult to see that all the addends on the right-hand side of \eqref{DV-ineq4} but the first converge to zero. Hence,
\begin{align}
\label{DV-ineq5}
& \limsup_{\varepsilon \to 0} \frac{1}{\varepsilon} \Big(V(x+\varepsilon \hat{n}) - V(x) \Big) \leq \E\bigg[\int_0^{\tau^{\star}} e^{-(\rho-\lambda)t} G_{\hat{n}}(X^{x,0}_t) \d t - e^{-(\rho-\lambda)\tau^{\star}}\bigg] \nonumber \\
& \leq \sup_{\tau \in \mathcal{T}} \E\bigg[\int_0^{\tau} e^{-(\rho-\lambda)t} G_{\hat{n}}(X^{x,0}_t) \d t - e^{-(\rho-\lambda)\tau}\bigg].
\end{align}

Combining \eqref{DV-ineq2} and \eqref{DV-ineq5} and using the fact that $V\in C^{1,\text{Lip}}(H)$ (cf.\ Proposition \ref{prop:regularityV}(v)) we obtain \eqref{eq:v} and thus complete the proof.
\end{proof}

We assume the next condition on the directional derivative $G_{\hat{n}}$.
\begin{assumption}
\label{ass:psiweak}
$G_{\hat{n}}\in C^1(H)$ and $|DG_{\hat{n}}|_H \leq K_{G_{\hat{n}}}$.
\end{assumption}

For our subsequent analysis, it is convenient to make an integration by parts and exploit the strong Markov property to write
$$V_{\hat{n}}(x) = -1 - \Phi(x)+ \sup_{\tau \in \mathcal{T}}\E\Big[e^{-(\rho-\lambda)\tau} \Phi(X^{x,0}_{\tau})\Big], \quad x \in H,$$
with 
\begin{equation}
\label{eq:Phi}
\Phi(x):=-\E\bigg[\int_0^{\infty} e^{-(\rho-\lambda)t}\Big(G_{\hat{n}}(X^{x,0}_t)  + \rho - \lambda\Big) \d t\bigg], \quad x \in H.
\end{equation}
Under Assumption \ref{ass:psiweak}, using \eqref{eq:mild-uncontr} and that  $|e^{t\mathcal{A}}|_{\mathcal{L}(H)} \leq e^{-\delta t}$, one finds for any $x_1, x_2 \in H$ that (recall $\lambda \leq - \delta <0$)
\begin{eqnarray}
\label{eq:stimaLipPhi}
|\Phi(x_2) - \Phi(x_1)|_H & \leq & \E\bigg[\int_0^{\infty} e^{-(\rho-\lambda)t} \big|G_{\hat{n}}(X^{x_2,0}_t) -G_{\hat{n}}(X^{x_1,0}_t)\big|_H \d t\bigg] \nonumber \\
   & \leq &  \E\bigg[\int_0^{\infty} e^{-(\rho-\lambda)t} \big|e^{t\mathcal{A}}(x_2-x_1)\big|_H \d t\bigg] \leq \frac{K_{G_{\hat{n}}}}{\rho - \lambda + \delta} |x_2- x_1|_H.
\end{eqnarray}
Actually, given that $G_{\hat{n}} \in C^1(H)$, it can be easily shown by direct calculations that $\Phi \in C^1(H)$.

With regard to those properties of $\Phi$, in order to further investigate the $C^1$-regularity of $V_{\hat{n}}$ it then suffices to consider
\begin{equation}
\label{eq:def-u}
U(x):=V_{\hat{n}}(x) + 1 + \Phi(x) = \sup_{\tau \in \mathcal{T}}\E\Big[e^{-(\rho-\lambda)\tau} \Phi(X^{x,0}_{\tau})\Big], \quad x \in H.
\end{equation}

\begin{proposition}
\label{prop:Uprelim}
Let Assumption \ref{ass:psiweak} hold. Then, one has
$$
|U(x_2) - U(x_1)|_H\leq  \frac{K_{G_{\hat{n}}}}{\rho - \lambda + \delta} |x_2- x_1|_H, \quad x_1, x_2 \in H.$$ 
\end{proposition}
\begin{proof}
For $x_1, x_2 \in H$, recalling \eqref{eq:mild-uncontr} and using \eqref{eq:stimaLipPhi} as well as  $|e^{t\mathcal{A}}|_{\mathcal{L}(H)} \leq e^{-\delta t}$ for any $t\geq0$, one has
\begin{eqnarray*}
|U(x_2) - U(x_1)| & \leq & \sup_{\tau \in \mathcal{T}}\E \Big[e^{-(\rho-\lambda)\tau}\big|\Phi(X^{x_2,0}_{\tau}) - \Phi(X^{x_1,0}_{\tau})\big|\Big] \nonumber \\
& \leq &  \frac{K_{G_{\hat{n}}}}{\rho - \lambda + \delta} \sup_{\tau \in \mathcal{T}}\E \Big[e^{-(\rho-\lambda)\tau} \big|e^{\tau\mathcal{A}}(x_2-x_1)\big|_H\Big] \leq \frac{K_{G_{\hat{n}}}}{\rho - \lambda + \delta} |x_2- x_1|_H.
\end{eqnarray*}
\end{proof}

By continuity of $U$, the stopping region $\{x\in H:\, U(x) =\Phi(x)\}$ is closed. Hence, by standard theory of optimal stopping (see, e.g., \cite[Ch.\ I, Sec.\ 2.2]{PeskirShir}, whose results hold for an underlying process taking values in a measurable space), one has $\P$-a.s.\ that
$$\tau^{\star}:=\inf\{t\geq 0: \ U(X^{x,0}_{t})=\Phi(X^{x,0}_{t})\}, \quad x \in H,$$
is optimal. 

\begin{proposition}[Dynamic Programming Principle for $U$]\label{DPPU}
For all stopping times $\theta\in \mathcal{T}$ we have
\begin{equation}\label{DPP}
U(x)=\sup_{\tau\in \mathcal{T}}\E\left[e^{-(\rho-\lambda)\tau} \mathds{1}_{{\tau<\theta}} \Phi(X^{x,0}_{\tau})+ e^{-(\rho-\lambda)\theta} \mathds{1}_{{\tau\geq \theta}} U(X^{x,0}_{\theta})\right],
 \ \ \ \ \forall x\in H.
\end{equation}
\end{proposition}
The proof of this result does not depend on the dimension of the underlying state space and proceeds along the same lines as in the finite-dimensional case.(see, e.g., \cite[Thm.\ 4.3]{Touzi} and references therein), upon exploiting the flow property of $X^{x,0}$ (see \eqref{eq-flow property} when $I\equiv 0$).

Based on the dynamic programming principle, one has that the differential problem associated to $U$ is the variational inequality
\begin{equation}
\label{VI}
\min\big\{((\rho-\lambda)-\mathcal{G})u, \ u-\Phi\big\}=0 \quad \text{on}\quad H,
\end{equation}
with $\mathcal{G}$ as in \eqref{eq:generator}, and for which, recalling the class \eqref{ClassX}, we now provide the definition of viscosity solution.

\begin{definition} 
\label{def:visc-OS}
\begin{enumerate}[(i)]
\item[]
\item We say that $u\in C(H)$ is a viscosity supersolution to \eqref{VI} at $x\in H$ if, for every $\varphi\in \mathcal{X}$ such that $0=u(x)-\varphi(x)=\min(u-\varphi)$, one has
$$\min\big\{((\rho-\lambda)-\mathcal{G})\varphi(x), \ (\varphi-\Phi)(x)\big\}\geq 0.$$
\item We say that $u\in C(H)$ is a viscosity subsolution to \eqref{VI} at $x\in H$ if, for every $\varphi\in \mathcal{X}$ such that $0=u(x)-\varphi(x)=\max(u-\varphi)$, one has
$$\min\big\{((\rho-\lambda)-\mathcal{G})\varphi(x), \ (\varphi-\Phi)(x)\big\}\leq 0.$$
\item We say that $u\in C(H)$ is a viscosity solution to \eqref{VI} at $x\in H$ if it is both a viscosity super- and subsolution.
\end{enumerate}
\end{definition}

\begin{remark}
\label{rem:OS-radial}
In this case, one could also include radial functions in the definition \ref{def:visc-OS} of viscosity solution of the considered variational inequality. We refrain from doing so for the sake of simplicity and because the goal of this section is to exploit the viscosity property -- specifically, the viscosity supersolution property -- using only test functions in $\mathcal{X}$ to establish a smoothness result, rather than to prepare a uniqueness result (see Remark \ref{rem:vis} and compare with Remark \ref{rem:on regularity of U}).
\end{remark}

\begin{lemma}
\label{lemma:visc} 
Let $x\in H$, $\varepsilon>0$, and let 
 $$\theta:=\inf\{t\geq 0: \ X^{x,0}_{t}\notin \mathcal{B}_{|\cdot|_H}(x,\varepsilon)\}\wedge 1,$$
with $\mathcal{B}_{|\cdot|_H}(x,\varepsilon):=\{y \in H:\, |y-x|_H  < \varepsilon\}.$
There exists $m_{o}>0$ such that
\begin{equation}\label{stun}
 \E\left[\int_{0}^{\theta \wedge\tau}e^{-(\rho-\lambda) t} \d t +\frac{e^{-(\rho-\lambda)\tau}}{\rho-\lambda}\mathds{1}_{{\tau<\theta}}\right]\geq m_{o},  \ \ \forall \tau\in\mathcal{T}.
\end{equation}
\end{lemma}

\begin{proof}
First of all, notice that
\begin{align*}
\label{stun}
& \E\left[\int_{0}^{\theta\wedge\tau}e^{-(\rho-\lambda) t} \d t +\frac{e^{-(\rho-\lambda)\tau}}{\rho-\lambda}\mathds{1}_{{\tau<\theta}}\right]\\
& =\E\left[\left(\int_{0}^{\theta}e^{-(\rho-\lambda) t} \d t\right) \mathds{1}_{{\tau\geq \theta}} + \left(\int_{0}^{\tau}e^{-(\rho-\lambda) t} \d t \right)\mathds{1}_{{\tau< \theta}}+\frac{e^{-(\rho-\lambda)\tau}}{\rho-\lambda}\mathds{1}_{{\tau<\theta}}\right]\\
& =\E\left[\frac{1-e^{-(\rho-\lambda)\theta}}{\rho-\lambda}\mathds{1}_{{\tau\geq \theta}} + \frac{1}{\rho-\lambda}\mathds{1}_{{\tau<\theta}}\right].
\end{align*}
Then assume, by aiming for a contradiction, that there exists a sequence $(\tau_{n})\subseteq \mathcal{T}$ such that 
$$\E\left[\frac{1-e^{-(\rho-\lambda)\theta}}{\rho-\lambda}\mathds{1}_{{\tau_{n}\geq \theta}} + \frac{1}{\rho-\lambda}\mathds{1}_{{\tau_{n}<\theta}}\right]\to 0.$$
This means that 
\begin{equation}\label{conv}
\E\left[\frac{1-e^{-(\rho-\lambda)\theta}}{\rho-\lambda}\mathds{1}_{{\tau_{n}\geq \theta}}\right]\to 0 \ \ \ \mbox{and} \ \ \  \E\left[\mathds{1}_{{\tau_{n}<\theta}}\right]\to 0.
\end{equation}
The second convergence above yields, passing to a subsequence still labeled in the same way, 
$$
\lim_{n} \mathds{1}_{\tau_{n}\geq \theta}=1 \ \mbox{a.s.}.
$$
But then, the dominated convergence theorem gives
$$
\E\left[\frac{1-e^{-(\rho-\lambda)\theta}}{\rho-\lambda}\mathds{1}_{{\tau_{n}\geq \theta}}\right]\to \E\left[\frac{1-e^{-(\rho-\lambda)\theta}}{\rho-\lambda}\right]>0,
$$
where the strict inequality in the formula above is clearly due to continuity of trajectories of the process $X^{x,0}$. Hence, we have found a contradiction with the first convergence in \eqref{conv} and the proof is thus complete.
\end{proof}

\begin{proposition} 
\label{prop:Uvisc}
$U$ is a viscosity solution to \eqref{VI} at all $x\in H$.
\end{proposition}

\begin{proof}
We follow the ideas of the proof in the finite-dimensional setting proposed by \cite[Thm.\ 5.2.1]{Pham}, but we simplify substantially the proof of the subsolution property thanks to Lemma \ref{lemma:visc}.
\smallskip

\emph{Supersolution property.} Recall \eqref{ClassX} and let $\varphi\in \mathcal{X}$ be such that 
$$0=U(x)-\varphi(x)=\min(U-\varphi).$$ Clearly, 
$$\varphi(x)=U(x)\geq \Phi(x).$$ Therefore, it remains to show that 
$$((\rho-\lambda)-\mathcal{G})\varphi(x)\geq 0.$$
 Setting
$$\tau_R:=\inf\{t\geq0:\, |X^{x,0}_t|_H \geq R\}$$ (with $\inf \emptyset = + \infty$) and letting $h>0$,  by {Proposition \ref{DPPU} with $\theta=\tau_{R}\wedge h$,} we have
$$U(x) \geq \E[e^{-(\rho-\lambda)(\tau_R\wedge h)} U(X^{x,0}_{\tau_R\wedge h})],$$ so that
$$
0\geq \E[e^{-(\rho-\lambda)(\tau_R\wedge h)} U(X^{x,0}_{\tau_R\wedge h})]-U(x)\geq  \E[e^{-(\rho-\lambda)(\tau_R\wedge h)} \varphi(X^{x,0}_{\tau_R\wedge h})]-\varphi(x) .
$$ 
On the other hand, Proposition \ref{Dyn1}, applied in the case of $I=\textbf{0}$ (the null control), yields 
$$\E[e^{-(\rho-\lambda) (\tau_R\wedge h)} \varphi(X^{x,0}_{\tau_R\wedge h})]= \varphi(x)+\E\left[\int_{0}^{\tau_R\wedge h}e^{-(\rho-\lambda) t}[(\mathcal{G}-(\rho-\lambda))\varphi](X^{x,0}_{t}) \d t \right].$$
Combining the last two display equations and dividing by $h$ we obtain
$$
\frac{1}{h}\E\left[\int_{0}^{\tau_R\wedge h}e^{-(\rho-\lambda) t}\Big[((\rho-\lambda)-\mathcal{G})\varphi\Big](X^{x,0}_{t}) \d t \right]\geq 0.
$$
We conclude by taking $h\to 0^{+}$ and applying the integral mean-value theorem and the dominated convergence theorem, since $\varphi \in \mathcal{X}$.
 
\smallskip

\emph{Subsolution property.} Let  $\varphi\in \mathcal{X}$ be such that $0=U(x)-\varphi(x)=\max(U-\varphi)$ and  assume, by contradiction, that there exists some $\eta>0$ such that  
 $$((\rho-\lambda)-\mathcal{G})\varphi(x)>2\eta \ \ \mbox{and} \  \  U(x)-\Phi(x)>\frac{2\eta}{\rho-\lambda }.$$
 By continuity, {there exists some $\varepsilon>0$ such that}
  $$((\rho-\lambda) -\mathcal{G})\varphi(y)>\eta \ \ \mbox{and} \  \  U(y)-\Phi(y)>\frac{\eta}{\rho-\lambda} \ \ \ {\forall y\in} \mathcal{B}_{|\cdot|_H}(x,\varepsilon),$$
where $\mathcal{B}_{|\cdot|_H}(x,\varepsilon):=\{y \in H:\, |y-x|_H<  \varepsilon\}.$
	
Let us define the stopping time $${\theta}:=\inf\{t\geq 0: \ X^{x,0}_{t}\notin \mathcal{B}_{|\cdot|_H}(x,\varepsilon)\}\wedge 1.$$ Then, using Dynkin's formula of Proposition \ref{Dyn1} (applied again in the case of $I$ being the null control) we obtain for each $\tau\in\mathcal{T}$
\begin{align*}
& \E\left[e^{-(\rho-\lambda)(\theta\wedge\tau)} U(X^{x,0}_{\theta\wedge\tau})\right]-U(x)\leq  \E\left[e^{-(\rho-\lambda)(\theta\wedge\tau)} \varphi(X^{x,0}_{\theta\wedge\tau})\right]-\varphi(x)\nonumber \\
& =\E\left[\int_{0}^{\theta\wedge\tau}e^{-(\rho-\lambda) t}\Big[\big(\mathcal{G}-(\rho-\lambda)\big)\varphi\Big](X^{x,0}_{t}) \d t \right],
\end{align*}
which in turn, thanks to Lemma \ref{lemma:visc}, implies
\begin{align*}\label{dsa}
 U(x)&\geq \E\left[\eta\int_{0}^{\theta\wedge\tau}e^{-(\rho-\lambda) t}\d t +e^{-(\rho-\lambda) (\theta\wedge\tau)} U(X^{x,0}_{\theta\wedge\tau})\right]\\&
 \geq \E\left[\eta\int_{0}^{\theta\wedge\tau}e^{-(\rho - \lambda)t}\d t +e^{-(\rho-\lambda) \tau} \mathds{1}_{{\tau<\theta}} \left(\frac{\eta}{\rho-\lambda}+\Phi(X^{x,0}_{\tau})\right)+ e^{-(\rho-\lambda)\theta} \mathds{1}_{{\tau\geq \theta}} U(X^{x,0}_{\theta})\right]\\
&=\eta \E\left[\int_{0}^{\theta\wedge\tau}e^{-(\rho-\lambda) t} \d t +	\frac{e^{-(\rho-\lambda) \tau}}{\rho-\lambda}\mathds{1}_{{\tau<\theta}} \right]+\E\left[e^{-(\rho-\lambda)\tau} \mathds{1}_{{\tau<\theta}} \Phi(X^{x,0}_{\tau})+ e^{-(\rho-\lambda) \theta} \mathds{1}_{{\tau\geq \theta}} U(X^{x,0}_{\theta})\right]\\
&\geq \eta m_{o}+\E\left[e^{-(\rho-\lambda) \tau} \mathds{1}_{{\tau<\theta}} \Phi(X^{x,0}_{\tau})+ e^{-(\rho-\lambda) \theta} \mathds{1}_{{\tau\geq \theta}} U(X^{x,0}_{\theta})\right].
\end{align*}
Taking the supremum over $\tau\in\mathcal{T}$ in the latter, we contradict  \eqref{DPP}, concluding the proof.
\end{proof}

\begin{remark}
\label{rem:uniqueness2}
Comments as those collected in Remark \ref{rem:uniqueness} apply to \eqref{VI}. The analogy of our setting with \cite{GS} is at this point even tighter, as now \eqref{VI} takes the form of an obstacle problem and thus involves a constraint on the solution itself (and not on its gradient as it was in the previous section). 

If one aims at a comparison theorem for \eqref{VI}, this might be proved by adapting the techniques of \cite{GS} in order to deal with the unbounded operator in the dynamics of the state process (not present in \cite{GS}). To that end, one should treat the unbounded term as in the regular control case by adding radial functions as test functions in the definition of viscosity solutions (see \cite[Ch.\,3]{FGS}). 

However, because our main aim is to provide regularity results for the optimal stopping problem under consideration, we refrain from studying the relevant question of uniqueness of the viscosity solution to \eqref{VI}, which is then left for future research.
\end{remark}

Thanks to Assumption \ref{ass:A}, we can introduce
\begin{equation}
\label{eq:norm-minus-1}
|x|_{-1}:= |\mathcal{A}^{-1}x|_{H},
\end{equation}

Then, in order to achieve the $C^1$-regularity of $U$, we strengthen the assumption on $G_{\hat{n}}$, by requiring the following, which will be standing in the rest of the section.
\begin{assumption}
\label{ass:psiweak-2}
There exists $\kappa_o>0$ such that $|G_{\hat{n}}(x)|\leq \kappa_o(1+|x|_{-1})$. Furthermore,
$G_{\hat{n}}$ is semiconcave with respect to the norm $|\cdot|_{-1}$; that is, there exists $\kappa_1>0$ such that
$$
\lambda G_{\hat{n}}(x)+(1-\lambda) G_{\hat{n}}(y) - G_{\hat{n}}(\lambda x+(1-\lambda) y)\leq \frac{\kappa_1}{2}\lambda(1-\lambda)|x-y|_{-1}^{2}, \ \ \forall x,y\in H, \ \lambda\in[0,1].
$$
\end{assumption}

\begin{remark}
\label{rem:G}
Assumptions \ref{ass:C}, \ref{ass:psiweak}, and \ref{ass:psiweak-2} are satisfied, e.g., if 
$$
{G}(x)=\frac{1}{2}(\langle x,h\rangle_{H})^2, \quad \text{or} \quad {G}(x)=\frac{1}{2}\langle Q x, x\rangle_{H}, \quad x \in H,
$$
with $h \in \mathcal{D}(\mathcal{A})$, and with $Q$ being positive semidefinite, symmetric and such that $Q\hat{n} \in \mathcal{D}(\mathcal{A})$.

Indeed, in these cases, 
$$
{G}_{\hat{n}}(x)= \langle x, h\rangle_{H} \langle h, \hat{n} \rangle_{H}, \quad \text{respectively} \quad {G}_{\hat{n}}(x)=\langle x, Q \hat{n} \rangle_{H}, \quad x \in H,
$$
so that ${G}_{\hat{n}}$ is clearly concave, belongs to $C^1(H)$ and it has sublinear growth with respect to $|\cdot|_H$. Moreover, in the first case, setting $k:=\mathcal{A}h$, one has 
\begin{align*}
|{G}_{\hat{n}}(x)|&\leq |h|_{H} \big|\langle x,h\rangle_{H}\big|=  |h|_{H} \big|\langle x,\mathcal{A}^{-1}k\rangle_{H}\big|=  |  h|_{H} \big|\langle x,\mathcal{A}^{-1} k\rangle_{H}\big|\\
&=  |h|_{H}\big|\langle \mathcal{A}^{-1}x,k\rangle_{H}\big|\leq  |h|_{H} |k|_{H}|x|_{-1}.
\end{align*}
Similarly, in the second case, setting $k:=\mathcal{A} Q\hat{n}$, one has 
\begin{align*}
|{G}_{\hat{n}}(x)|&\leq |\langle x, Q\hat{n}\rangle_{H}|= |\langle x,\mathcal{A}^{-1}k\rangle_{H}|= |\langle x,\mathcal{A}^{-1} k\rangle_{H}|\\
&= |\langle \mathcal{A}^{-1}x,k\rangle_{H}| \leq |k|_{H}|x|_{-1}.
\end{align*}
\end{remark}

Recall \eqref{eq:stochintegr} and \eqref{eq:norm-minus-1}. By \eqref{eq:mild-uncontr}, it holds 
$$|\mathcal{A}^{-1} X^{x,0}_s|_H \leq |e^{t\mathcal{A}}\mathcal{A}^{-1} x|_H + |\mathcal{A}^{-1} W^{\mathcal{A},\sigma}_s|_H.$$ 
Then, the contraction property of the semigroup $e^{t\mathcal{A}}$, estimate \eqref{estimatesup-W}, and the fact that the norm $|\cdot|_{-1}$ is clearly dominated by the norm $|\cdot|_H$ give
\begin{equation}
\label{estimatesup-1}
\E\left[\sup_{s\geq 0} |X^{x,0}_{s}|_{-1}\right]\leq \bar{\kappa} (1+|x|_{-1}), \ \ \ \ \forall x\in H, 
\end{equation}
for some $\bar{\kappa}>0$.

One then has the following preliminary result.

\begin{proposition}
\label{prop:U-1}
 $U$ is semiconvex with respect to the $|\cdot|_{-1}$ norm and there exists $\widehat{\kappa}_o>0$ such that
\begin{equation}
\label{eq:stimecreascitaU}
|U(x)|\leq \widehat{\kappa}_o(1+|x|_{-1}).
\end{equation} 
Furthermore, $U$ is $|\cdot|_{-1}$-locally Lipschitz continuous.
\end{proposition}
\begin{proof}
Using \eqref{eq:mild-uncontr} and that  $|e^{t\mathcal{A}}|_{\mathcal{L}(H)} \leq e^{-\delta t}$, one easily finds from \eqref{eq:Phi} that the semiconcavity with respect to the $|\cdot|_{-1}$ norm of ${G}_{\hat{n}}$ implies semiconvexity of $\Phi$ with respect to the $|\cdot|_{-1}$ norm as well. Hence, given that $x \mapsto X^{x,0}_{\cdot}$ is linear, we find that $U$ is semiconvex with respect to the $|\cdot|_{-1}$ norm being supremum of semiconvex functions (see \cite[Prop.\ 2.1.5]{CannarsaSinestrari}, whose proof does not suffer the dimensionality of the state space).

By \eqref{eq:Phi}, one also has that the growth condition on ${G}_{\hat{n}}$, together with \eqref{estimatesup-1}, imply that $|\Phi(x)|\leq K_{\Phi}(1+|x|_{-1})$, for a suitable $K_{\Phi}>0$. The latter, together with \eqref{eq:stimecreascitaU}, in turn gives for $\widehat{\kappa}_o>0$
$$|U(x)| \leq \sup_{\tau \in \mathcal{T}}\E\big[|\Phi(X^{x,0}_{\tau})|\big] \leq K_{\Phi}\big( 1 + \E\big[ \sup_{s\geq 0}|X^{x,0}_{s}|_{-1}\big]\big) \leq \widehat{\kappa}_o(1+|x|_{-1}).$$

Finally, the last claim is due to \cite[Cor.\,2.4, p.\,12]{EkelandTemam} and to the fact that, being $U$ semiconvex with respect to $|\cdot|_{-1}$ norm,  one can write
$$U(x) = U_0(x) - \frac{\widehat{\kappa}_1}{2}|x|^2_{-1}, \quad x \in H,$$
for some $\widehat{\kappa}_1>0$ and $U_0$ convex. 
 \end{proof}

In the rest of the paper, we are going to study the regularity properties of the (sub)gradient of $U$,  that we denote by $D^{-}U$.

\begin{proposition}
\label{prop:gradient}
$
D^- U(x)\subseteq \mathcal{D}(\mathcal{A})$
 at all $x\in H$.
 \end{proposition}

\begin{proof}
Recall that, by semiconvexity, we can write
$$U(x) = U_0(x) - \frac{\widehat{\kappa}_1}{2}|x|^2_{-1}, \quad x \in H,$$
for some $\widehat{\kappa}_1>0$ and $U_0$ convex, so that, for any $x\in H$,
$$D^-U(x) = D^-U_0(x) - \widehat{\kappa}_1 {\left(\mathcal{A}^{-1}\right)^{2}} x.$$
Hence, without loss of generality, up to replace $U$ by $U_0$, we  work in the rest of this proof under the assumption that $U$ is convex. 
%

Let $x\in H$, $(h_{n})\subseteq  \mathcal{D}(\A)$ be such that $h_{n}\to 0$ with respect to $|\cdot|_{H}$, and set $\hat h_{n}:=\A h_{n}$, so that also $|\hat h_{n}|_{-1}\to 0$. For $p\in D^-U(x)$, by convexity of $U$ we have
$$
U(x+\hat h_{n})-U(x)\geq \langle p,\hat h_{n}\rangle_{H},
$$
which rewrites as
\begin{equation}
\label{Uconvex1}
U(x+\hat h_{n})-U(x)\geq \langle p,\A  h_{n}\rangle_{H}.
\end{equation}
Since $U$ is $|\cdot|_{-1}$-continuous by Proposition \ref{prop:U-1}, the linear functional 
$$
T_{x}: \mathcal{D}(\A)\to \R, \ \ \ \ T_{x}(h)=  \langle p,\A  h\rangle_{H}
$$
is $|\cdot|_{H}$-upper semicontinuous by \eqref{Uconvex1}; that is, $\limsup_{n\rightarrow \infty} T_{x}(h_n)\leq 0$. Taking now the sequence $(-h_{n})\subseteq  \mathcal{D}(\A)$ and exploiting the linearity of $T_{x}$, it also holds that $T_{x}$ is $|\cdot|_{H}$-lower semicontinuous. Hence, it is $|\cdot|_{H}$-continuous, which means that $p\in\mathcal{D}(\A)$. 
%
\end{proof}

\begin{proposition}
\label{prop:regularity}
Let $x\in H$, $\hat h\in K$, and set $ h:=\sigma\hat h$. Then, $U$ is differentiable at $x$ along the direction $h$.
\end{proposition}
\begin{proof} We assume, without loss of generality, that $U(x)=0$ (indeed, this can be always achieved by a translation) and we notice that we can write
\begin{equation}\label{2024-07-24:00}
U(y) = U_0(y) - \frac{\widehat{\kappa}_1}{2}|y-x|^2_{-1}, \quad y \in H,
\end{equation}
for some $\widehat{\kappa}_1\geq 0$ and $U_0$ convex. 

We now argue by contradiction and assume that $U$, hence $U_0$, is not differentiable along the direction $h:=\sigma\hat h$, with $\hat h\in K$. By convexity of $U_0$, there exist $p_{1},{p}_{2}\in D^{-}U_0(x)$ such that
\begin{equation}\label{xxc}
\langle p_{1},h\rangle_{H}\neq \langle {p}_{2}, h\rangle_{H}.
\end{equation}
By \eqref{xxc} and since $h=\sigma\hat h$, we have 
\begin{equation}\label{xxcc}
 \sigma^* (p_1-p_2)\neq 0.
 \end{equation}
Again using  convexity of $U_{0}$,  
we get 
\begin{equation}\label{2024-07-24:06}
U_0(y)\geq \langle p_{1},y-x\rangle_{H}\vee \langle  p_{2},y-x\rangle_{H}, \ \ \ \forall y\in H.
\end{equation}
Note that, by Proposition \ref{prop:gradient}, we have $p_{1},{p}_{2}\in  \mathcal{D}(\mathcal{A})$.
For  $n\in \mathbb{N}$, let $\rho_n\colon \mathbb{R}\rightarrow \mathbb{R}$ be a function with the following properties:
\begin{subequations}
  \begin{align}
    \label{2024-07-28:01}
    &\rho_n(0)=0\\
    \label{2024-07-28:02}
    &\rho_n(t)=\rho_n(-t)\qquad\qquad  \forall t\in \mathbb{R}\\
    \label{2024-07-28:03}
    &  0\leq \rho_n(t)\leq |t|\qquad\qquad \forall t\in \mathbb{R}\\
    \label{2024-07-28:04}
    &   \rho_n\in C^2(\mathbb{R})\\
    \label{2024-07-28:05}
    &   \rho_n'(0)=0\\
    \label{2024-07-28:06}
    &     \rho_n''(0)= n
  \end{align}
\end{subequations}
and define $\hat\varphi_n\colon H\rightarrow\mathbb{R}$,
\begin{equation*}
  \hat\varphi_n(y)\coloneqq \rho_n
  \left(
    \langle p_1,y-x\rangle_{H}
    \vee
    \langle p_2,y-x\rangle_{H}
    -\frac{1}{2}
    \langle p_1+p_2,y-x\rangle_{H}
  \right)
  +
  \frac{1}{2}
  \langle p_1+p_2,y-x\rangle_{H}\qquad \forall y\in H.
\end{equation*}
Noticing that
\begin{equation}\label{aaaz}
  \langle p_1,y-x\rangle_{H}
  \vee
  \langle p_2,y-x\rangle_{H}
  -\frac{1}{2}
  \langle p_1+p_2,y\rangle_{H}
  =\frac{1}{2}
  \left| \langle p_1-p_2,y\rangle_H\right|,
\end{equation}
we have
\begin{equation*}
  \begin{split}
    \hat\varphi_n(y)
    =&
    \rho_n \left(\frac{1}{2}
     \left| \langle p_1-p_2,y-x\rangle_{H}\right|
   \right)   +  \frac{1}{2} \langle p_1+p_2,y-x\rangle_{H}\\
   =&\mbox{(by \eqref{2024-07-28:02})}\\
   =&
   \rho_n\left(\frac{1}{2}
       \langle p_1-p_2,y-x\rangle_{H}
   \right)   +  \frac{1}{2} \langle p_1+p_2,y-x\rangle_{H}\qquad \forall y\in H.
  \end{split}
\end{equation*}
The last expression tells us that
$\hat\varphi_n\in C^2(H)$, and
$$
D\hat\varphi_n(y)=\frac{1}{2}\left[\rho_{n}'\left(\frac{1}{2}
       \langle p_1-p_2,y-x\rangle_{H}\right)+1\right] (p_1+p_2).
$$
Since $p_1,p_{2}\in \mathcal{D}(\A)$, this
 shows that $\hat\varphi_{n}\in \mathcal{X}$. 
Moreover, 
\begin{subequations}
  \begin{equation}
    \label{2024-07-28:08}
    D\hat\varphi_n(x)=\frac{1}{2}(p_1+p_2) \in\mathcal{D}(\mathcal{A}),
  \end{equation}
  \begin{equation}
    \label{2024-07-28:09}
    \langle D^2\hat\varphi_n(x)y,z
    \rangle_{H}
    =\frac{1}{4}\rho''_n(0)
    \langle
    p_1-p_2,y
    \rangle_{H}   \langle
    p_1-p_2,z
    \rangle_{H}\qquad \forall y,z\in H.
  \end{equation}
\end{subequations}
Furthermore, for all $y\in H$,
\begin{equation}\label{2024-07-28:07}
  \begin{split}
      \hat\varphi_n(y)-\frac{\widehat{\kappa}_{1}}{2}|y-x|_{-1}^2\ \leq\ & \mbox{(by \eqref{2024-07-28:03} \, \mbox{and}  \, \eqref{aaaz})}
  \ \leq \ 
  \langle p_1,y-x\rangle \vee \langle p_2,y-x\rangle_{H}-\frac{\widehat{\kappa}_{1}}{2}|y-x|_{-1}^2\\
  \ \leq\ &
  \mbox{(by \eqref{2024-07-24:06})}
  \ \leq \ U_0(y)-\frac{\widehat \kappa_{1}}{2}|y-x|_{-1}^2\\
  \ =\  &  \mbox{(by \eqref{2024-07-24:00})}
  \ = \ U(y) .
\end{split}
\end{equation}
Define
\begin{equation*}
  \varphi_n(y)\coloneqq \hat\varphi_n(y)-\frac{\widehat{\kappa}_{1}}{2}|y-x|_{-1}^2 \qquad \forall y\in H.
\end{equation*}
Then $\varphi_{n}\in\mathcal{X}$ and, by \eqref{2024-07-28:08} and \eqref{2024-07-28:09}, we have
\begin{subequations}
  \begin{equation}
    \label{2024-07-28:10}
    D\varphi_n(x)=\frac{1}{2}(p_1+p_2)- \widehat{\kappa}_{1}\left(\A^{-1}\right)^{2}(y-x). 
  \end{equation}
\end{subequations}
Also, by 
\eqref{2024-07-28:09} and recalling \eqref{2024-07-28:06},
\begin{equation}\label{2024-07-28:13}
  \begin{split}
    \mbox{Tr}
\left[  \sigma\sigma^*D^2\varphi_n(x)\right]
=&
\frac{1}{4}n
  |
  \sigma^* (p_1-p_2)|_{H}^2+
\widehat{\kappa}_1\mbox{Tr}\left[\sigma\sigma^{*}\A^{-1}\right].
\end{split}
\end{equation}

\noindent By \eqref{2024-07-28:07}, $U-\varphi_n$ attains a minimum at $0$.
Because $U$ is a viscosity supersolution at $x$ of \eqref{VI} (cf.\ Proposition \ref{prop:Uvisc}) and recalling \eqref{2024-07-28:13} and \eqref{xxcc}, one finally has
\begin{align*}
0&\leq \big[\big((\rho-\lambda)-\mathcal{G}\big)\varphi_{n}\big](x)=- \frac{1}{2}\mbox{Tr}\left[\sigma\sigma^{*}D^{2}\varphi_{n}(x)\right]- \langle x,\mathcal{A} D\varphi_{n}(x)\rangle_{H}\\
&\leq - \frac{1}{4}n
  |
  \sigma^* (p_1-p_2)|_{H}^2 + \frac{1}{2}\widehat \kappa_1\mbox{Tr}\left[\sigma\sigma^{*}\A^{-1}\right] + \frac{1}{2}\langle \mathcal{A}(p_1+p_2), \,x\rangle_{H}\to-\infty \quad \text{as}\,\,n\to \infty, 
\end{align*}
which reaches the desired contradiction and thus completes the proof.
\end{proof}
\color{black}

\begin{remark}
\label{rem:on regularity of U}
It is worth noticing that the proof of Proposition \ref{prop:regularity} only exploits the viscosity supersolution property of $U$, which is actually the easiest part to be shown in the proof of Proposition \ref{prop:Uvisc}.
\end{remark}


\begin{proposition}
\label{prop:improved}
Let  $\mathrm{R}(\sigma)$ denote the range of $\sigma$. If 
$\overline{\mathrm{R}(\sigma)}=H,$ then $U$ is (Fr\'echet) differentiable  at all  $x\in H$ and  $DU(x)\in \mathcal{D}(\mathcal{A})$. Moreover, $DU\in C(H;H)$.
\end{proposition}

\begin{proof}
By following the arguments developed at the beginning of Proposition \ref{prop:gradient}, we can here assume that $U$ is convex.
Then, due to Proposition \ref{prop:regularity} and convexity of $U$, we know that, for each {$h\in\mathrm{R}(\sigma)$}, the set $\{\langle h, p\rangle_{H}:\, p\in D^{-}U(x)\}$ is a singleton. Since $\overline{\mbox{R}(\sigma)}=H$, this implies that   $D^{-}U(x)$ must be a singleton too.  Again by convexity of $U$, it follows that $U$ is differentiable at $x\in H$. Furthermore, using Proposition \ref{prop:gradient}, we conclude that $DU(x)\in \mathcal{D}(\mathcal{A})$. 

{Finally,  the last assertion follows from the Corollary at p.\,20 in \cite{Phelps}.}
%
%
%
%
\end{proof}
  \color{black}


\section{Second-order smooth-fit for $V$}
\label{sec:mainres}

Thanks to Proposition \ref{prop:improved} and Theorem \ref{thm:connection} we finally achieve the second-order smooth-fit of $V$ in the direction of $\hat{n}$. More precisely, we have the following corollary.

\begin{corollary}
\label{cor:2ndsmooth}
Under the assumptions of Proposition \ref{prop:improved}, $V\in C^{1,\text{Lip}}(H)$ and $V_{\hat{n}} \in C^1(H)$.
\end{corollary}

The second-order regularity of the value function has been a fundamental aspect of singular stochastic control theory since its beginning. The underlying idea is that such smoothness of the value function lays the groundwork for characterizing the problem's free boundary and, consequently, for constructing an optimal control of reflection type.

In one-dimensional or two-dimensional fully degenerate stationary settings, where a guess-and-verify approach can be employed, imposing a suitable second-order regularity on the solution of the variational inequality enables the unique determination of certain otherwise free parameters. This, in turn, allows for the identification of the value function of the problem and the free boundary at which the state process should be optimally reflected (see \cite{Al00,FK21,JS,MZ07}, among many others).

The validation of a second-order smooth-fit property in multiple dimensions still requires verification on a case-by-case basis. This is well described in the introduction of the seminal work by S.E.\ Shreve and H.M.\ Soner \cite{SoSh2}: \emph{"An important question is whether the principle of smooth fit can be expected to apply to multidimensional singular control problems, or is it strictly a one-dimensional phenomenon. Karatzas and Shreve \cite{KS86} suggested that it might apply in higher dimensions. [...] Our discovery of a $C^2$-value function provides strong support for the belief in a widely applicable principle of smooth fit. Nevertheless, the argument of this paper depends heavily on the fact that only two dimensions are involved [...], and we have not found a way to obtain a similar result in higher dimensions."}

In suitable multiple-dimensional frameworks, second-order regularity of the derivative of the value function of the singular stochastic control problem along the direction of the control process has been obtained more recently through its relation to optimal stopping (see, e.g., \cite{ChHaus,DianettiFerrari,Ferrari18}). Our result is thus situated within this body of literature and actually provides, for the first time, the validation of a second-order smooth-fit property in an infinite-dimensional setting.

\newpage

\section{Two applications in Economics}
\label{sec:applications}

Here we discuss two economic models that can embedded into our setting.

\subsection{An irreversible investment problem into energy capacity}
\label{ex:1}

 Let $\mathcal{O}$ be an open, simply connected, and bounded subset of $\mathbb{R}^n$, $n<4$, with smooth boundary. We endow it with the Lebesgue measure $\mu$ on the Borel $\sigma$-algebra of $\mathcal{O}$, and consider the Hilbert space $H=L^2(\mathcal{O};\mu)$.
 
 Within this mathematical framework, we consider a  company which  has to deliver  energy   to  the locations of $\mathcal{O}$. 
We assume that, in absence of any investment by the company,  the energy supply at time $t$ and location $\xi$ -- denoted by  $E^0(t,\xi)$ --  evolves   according to the parabolic PDE 
\begin{equation}\label{PDEneumann}
\begin{cases*}
\displaystyle{\frac{\partial E^0}{\partial t}(t,\xi) = \Delta E^0(t,\xi) - \delta E^0(t,\xi), \ \ \ \ (t,\xi)\in\mathbb{R}_+\times \mathcal{O},} \\
E^0_{}(0,\xi)=e(\xi), \quad \xi \in \mathcal{O},\\
\displaystyle{\frac{\partial E^0}{\partial \bf n} (t,\xi)=0, \quad  (t,\xi)\in \mathbb{R}_+\times \partial \mathcal{O},}
\end{cases*}
\end{equation}
where $\Delta$ is the Laplacian operator over $\mathcal{O}$, $\delta>0$ is a depreciation factor,  $e(\xi)$ is the initial value of the energy supply at location $\xi$, and $\bf n$ denotes the unitary outer normal vector at the boundary of $\mathcal{O}$; the Neumann boundary condition  at $\partial \mathcal{O}$ models the fact that we assume there is no flux of energy at the boundary.
Under suitable conditions on the domain, 
the above PDE is well posed as an abstract evolution equation in $H$. Precisely (see, e.g., \cite[Sec.3.1]{Lunardi}), defining the operator
$$
\mathcal{L}: \mathcal{D}(\mathcal{L}) \subseteq H\to H,
$$
where 
$$
\mathcal{D}(\mathcal{L})=W^{2,2}_0(\mathcal{O}):=\bigg\{f\in W^{2,2}(\mathcal{O}): \ \frac{\partial f}{\partial \bf n}=0  \ \mbox{on} \ \mathcal{O}\bigg\},  \ \ \ \mathcal{L}f:=\Delta f-\delta f,
$$
one has that $\mathcal{L}$ generates a strongly continuous semigroup of linear operators in $H$ and \eqref{PDEneumann} can be rewritten 
in abstract terms as
$$
\d E^0_t = \mathcal{L}E^0_t \d t,
 \ \ \ E_{0}=e.
$$
Assume now that the company can implement irreversible 
 investment strategies $I_t(\xi)$ in order to adjust the production capacity. 
 The energy supply over the region, i.e.\  the spatial process $E_t^I(\xi)$, then evolves  according to the controlled abstract evolution equation 
$$
\d E^I_t = \mathcal{L}E^I_t \d t + \d I_t, \ \ \ \ 
E^I_{0^-}=e. 
$$
Next, 
we model the total demand of energy as a spatial process $A_t(\xi)$ evolving according to the SPDE
$$ \d A_t(\xi) =  \mathcal{B} A_t(\xi)  \d t - \sigma \d W_t(\xi), \quad A_{0}(\xi)=a(\xi),$$
for some initial maximal demand $a \in H_+$, some bounded nonnegative self-adjoint linear operator $\mathcal B \in \mathcal{L}(H)$ such that $\mathcal{L}-\mathcal{B}$ satisfies Assumption \ref{ass:A2}, and for $\sigma$ and $W$ satisfying the requirements of Section \ref{sec:setting}. 
The set of admissible investment strategies is therefore naturally modeled by the class $\mathcal I$ (cf.\ \eqref{set:S}).
Moreover, we model 
the inverse demand function for energy assuming that  it depends on the location $\xi \in \mathcal{O}$ and is linear in the quantity $E^I_t(\xi)$ that is being delivered; that is,
$$ p_t(\xi, E^I_t(\xi))= A_t(\xi) - B(\xi) E^I_t(\xi),$$ for some function $B(\xi) >0$.
 Here, $A_t(\xi)/B(\xi)$ is  the maximal possible demand at location $\xi$ at time $t$. We set $B \equiv 1$ in the following (just for simplicity). Then, 
 the 
total surplus at a given time $t$ and location $\xi$ is given by
 $$U_t(\xi)= \int_0^{E_t(\xi)} p_t(\xi, z) \d z = A_t(\xi)  E_t(\xi)- \frac{1}{2}\left(E_t(\xi)\right)^2,$$
so that the overall total surplus at time $t$ is
$$ S_t=\int_D U_t(\xi) \mu(\d\xi)= \langle A_t , E^I_t\rangle_{H} - \frac{1}{2} \langle E^I_t , E^I_t\rangle_{H}
= - \frac{1}{2} \langle E^I_t - A_t, E_t- A_t\rangle_{H} + \frac{1}{2} \langle A_t, A_t \rangle_H.$$ 
As the last term is independent of the control $I$, it is irrelevant in the optimization process we are going to define and we  ignore it in the following.

Set now $X^I_t := E^I_t - A_t$ and  $\mathcal{A}:=\mathcal{L}  - \mathcal{B}$. Then we have
\begin{equation}
\label{eq:Xexample1}
\d X^I_t = \mathcal{A} X^I_t \d t + \sigma \d W_t + \d I_t, \quad X^I_{0^-}=e-a =:x,
\end{equation}
The latter controlled SPDE satisfies our Assumptions \ref{ass:A} and \ref{ass:A2}. 

For an intertemporal discount factor $\rho>0$, the energy producer aims at maximizing the total expected surplus, net of the investment costs. Assuming that the cost of investment may vary with location $\xi$ and that it is modeled by a local price $q\in H_+$, bounded away from zero -- that is, there exists $q_o>0$ such that $q (\xi)\geq q_o$ for every $\xi\in \mathcal{O}$ -- the company's optimization problem is
$$\sup_{I \in \mathcal{I}}\E\bigg[\int_0^{\infty} e^{-\rho t} \Big(-\frac{1}{2}|X_t|^2_H - \langle q, \d I_t\rangle_H\Big)\bigg].$$
The latter is clearly equivalent to minimizing the cost functional \eqref{eq:costfunct} with $G(x)= |x |^2_{H}$.

\subsection{An energy balance climate model with human impact}
\label{ex:2}
 Let us consider a one-dimensional Energy Balance Climate Model with Human Impact (see, e.g., \cite{Brocketal}; the basic climate model dates back to Gerald North, see \cite{North}).  

We first describe quickly the basics of an energy balance climate model. The earth's temperature is taken to be the result of incoming radiation from the sun and outgoing radiation through reflection. 
We consider temperature on the  hemisphere,  modeled by the half-circle that we identify with the interval $D:=[-1,1]$, where      $\xi\in[-1,1]$ is the sine of latitude. $\mathcal{M}$ is the Borel $\sigma$-algebra on $D$, and $\mu$ is taken to be the Lebesgue measure on $(D, \mathcal{M})$.
 
We now model the temperature evolution $T_t(\xi)$ over time $t$ at location $\xi \in D$.     
    The incoming radiation at $\xi$ is denoted by $R(\xi)=Q S(\xi) \alpha(\xi)$. Here, $S(\xi)$ is the solar energy arriving at latitude $\xi$, $Q$ is the solar  constant divided by $4$, and $\alpha(\xi)$ describes the amount of heat absorbed at location $\xi$ (co-albedo); in general, it depends on temperature and location, but here as in \cite{North} we assume that it is just a function of $\xi \in D$. 
    The outgoing infrared radiation at location $\xi$ is linear in temperature, say $\gamma + \eta T_t(\xi)$, for two constants $\gamma$ and $\eta>0$. On the surface, we have a typical heat transport that is modeled via the  second derivative $\frac{\d^2}{\d \xi^2}$, or more generally, by  the operator 
$$(\mathcal{B} f) (\xi) := \frac{\d}{\d \xi} \left (D(\xi) \frac{\d}{\d \xi} f(\xi)\right),$$ for some diffusion coefficient $D$ driving the heat transport.  
The overall resulting energy-balance operator  
    \begin{equation}
    \label{EqnClimate}
        \big(\mathcal{Q} f\big)(\xi) := Q S(\xi)  \alpha(\xi)  - \gamma - \eta f(\xi)  + \big(\mathcal{B}f\big)(\xi)
    \end{equation}
    describes the energy balance of the earth without human impact.
    The equilibrium temperature distribution $T^{\star}(\xi)$ is given by the solution to the partial differential equation $\mathcal{Q} T^{\star}=0$, subject to appropriate boundary conditions (for instance, both zero Neumann or periodic boundary conditions can be chosen).
    
    We now add human impact due to carbon emissions.
   Let global cumulative human carbon emissions be described by the (real-valued) process $\nu \in \mathcal{S}$ (cf.\ \eqref{set:S-new}).

   The temperature evolution at time $t$ is then described by  
   \begin{equation}
       \label{EqnBalance}
       \d T_t    =  \mathcal{Q} T_t   \d t + \sigma \d W_t   +   \mathbf{1}\d\nu_t 
    , \quad T_{0^-}=x \in H,
   \end{equation} 
    with $\mathbf{1}$ being the unitary vecor in $H$, and with $\sigma$ and $W$ as specified in Section \ref{sec:setting}. In particular, the Brownian motion $W$ takes care of noise and unmodeled influences. 
    
    The dynamics (\ref{EqnBalance}) does not fit exactly into our setting because of the constant (in time)  drift term 
    $$b(\xi):=Q S(\xi)  \alpha(\xi)-\gamma,$$ but this problem can be easily fixed. 
Define 
$$X_t:= T_t-T^{\star}, \ \ \ \ \mathcal{A} f :=  - \eta f  + \mathcal{B}f,$$
so that
$$\mathcal{Q}f=\mathcal{A}f+b.$$
Recalling that $T^{\star}$ is an equilibrium distribution for the temperature, we may write 
   \begin{align*} 
    \d X_t& = \d (T_t-T^{\star})  = \mathcal{Q} T_t\d t -\mathcal{Q}T^{\star}\d t+  \sigma \d W_t    +   \mathbf{1} \d\nu_t \\
    &= \mathcal{A} T_t \d t -\mathcal{A} T^{\star}\d t+ \sigma \d W_t   +   \mathbf{1} \d\nu_t \\
    &= \mathcal{A} X_t \d t+ \sigma \d W_t   +   \mathbf{1} \d\nu_t 
    \end{align*}
and we are   back in our setting. In particular, the operator $\mathcal{A}$ with the aforementioned boundary conditions satisfies Assumptions \ref{ass:A} and \ref{ass:A2} (for the null Neumann boundary conditions one, see, e.g., \cite[Sec.3.1]{Lunardi}; for the periodic ones, see Section 5.2 in \cite{FFRR}).

Assume now that a decision maker (for instance, in this context, the United Nations) has an ideal profile of  temperature $\widehat{T}$ in mind; this could be the pre-industrial equilibrium temperature $T^{\star}$ or a temperature distribution in its vicinity.
The planner thus aims at minimizing the  average square distance $|T- \widehat{T}|^2_H$ to that ideal temperature and measures the cost of investment into capacity by some price $q>0$. Then, the resulting minimization problem in terms of the controlled process $X^{\nu}$ is (cf.\ \eqref{set:S})
$$\inf_{\nu \in \mathcal{S}}\E\bigg[\int_0^{\infty} e^{-\rho t} \Big( \big|X^{\nu}_t + T^{\star} - \widehat{T}\big|^2_H \d t +  q \d\nu_t \Big) \bigg],$$
for some intertemporal discount rate $\rho>0$. This falls into our setting for $G(x)= |x+ T^{\star} - \widehat{T}\big|^2_H$.
Such a specification greatly simplifies the full economic model which would be beyond the scope of the current paper. Compare \cite{Brocketal14} for an attempt in that direction.
\medskip

\textbf{Acknowledgements.}  
Funded by the \emph{Deutsche Forschungsgemeinschaft} (DFG, German Research Foundation) - Project-ID 317210226 - SFB 1283. Salvatore Federico is grateful to Andrzej Swiech for the valuable discussions on viscosity solutions to variational inequalities in Hilbert spaces that led to the formulation of Remarks \ref{rem:uniqueness} and \ref{rem:uniqueness2}.


\appendix

\section{Proof of Proposition \ref{DPPOC}}
\label{sec:appendixA}

\renewcommand{\theequation}{A-\arabic{equation}}

We limit ourselves to provide a sketch of the proof, given that this follows closely from arguments used in \cite[Ch.\ 2 and Sec.\ 3.6.2]{FGS} (for the infinite-dimensional case with regular controls) and in \cite{DeAM} (for the finite-dimensional case with singular controls).

To this end, we let $t\geq0$, we take a cylindrical Wiener process $(W_s)_{s \in [t,\infty)}$ on $(\Omega, \mathcal{F},\mathbb{F},\mathbb{P})$ such that $W_t=0$ a.s., and, for $s\geq t$, we set $\mathcal{F}^{t,0}_s:=\sigma(W_r,\, t\leq r \leq s)$ and denote by $\mathcal{F}^{t}_s$ its completion with $\mathbb{P}$-null sets. We then write $\mathbb{F}^{t,0}:=(\mathcal{F}^{t,0}_s)_{s \in [t,\infty)}$ and $\mathbb{F}^{t}$ for the augmented filtration.

In analogy with \eqref{set:S-new} and \eqref{setI0}, recalling \eqref{setTHETA}, we let 
\begin{eqnarray}
\label{set:S-new-t}
\mathcal{S}_t&:=& \{\nu: \Omega \times [t,\infty) \to [0,\infty):\,\text{$\nu_{\cdot}$ is}\,\,\F^t-\text{adapted and such that}\,\, s \mapsto \nu_s \nonumber \\
&& \hspace{1.5cm} \text{is \ c\`adl\`ag and nondecreasing}\}.
\end{eqnarray}
In the sequel, we set $\nu_{t^-}:=0$ for any $\nu \in \mathcal{S}_t$ and define
\begin{eqnarray}
    \mathcal{I}_t& :=& \Big\{(\vartheta, \nu):\Omega \times [t,\infty) \to \Theta \times [0,\infty): \ \vartheta_\cdot \ \mbox{is} \ \F^t-\mbox{adapted},\,\, \nu_\cdot\in \mathcal{S}_t\,\,\, \text{and} \nonumber \\
    && \E\bigg[\Big|\int_{t^-}^T |{\vartheta}_s|_H\, \d \nu_s\Big|^2\bigg]  +  \E\bigg[\int_{t^-}^T |{\vartheta}_s|^2_H\, \d \nu_s\bigg] < \infty \,\, \forall T>0 \Big\}.
    \label{setI0-t}
\end{eqnarray}
Finally, we set (cf.\ \eqref{eq:admissiblesetfinal-t})
\begin{equation}
    \label{eq:admissiblesetfinal-t}
\hat{\mathcal{I}}_{t}:= \left\{I=(\vartheta,\nu)\in\mathcal{I}_{t}: \  \E\bigg[\int_{t^-}^{\infty} e^{-\rho (s-t)}\d \nu_{s}\bigg]<\infty\right\}.
\end{equation}

Now, for $I:=(\vartheta,\nu)\in\hat{\mathcal{I}_t}$ and $x \in H$, we let $(X^{t,x,I})_{s\in[t,\infty)}$ be the unique mild solution to
\begin{equation}
\label{eq:state-t}
\d X^{I}_{s}=\mathcal{A}X^{I}_{s}\d t+\sigma \d W_{s} + \d I_{s}, \quad s \geq t,  \ \ \ X_{t^{-}}^{I}=x \in H;
\end{equation}
that is,
\begin{equation}
\label{mild-t}
X_{s}^{t,x,I}=e^{(s-t)\A}x+\int_{t^-}^{s}e^{(s-r)\A}\sigma \d W_{r}+ \int_{t^-}^{s}e^{(s-r)\A}\d I_{r}, \quad s\geq t.
\end{equation}
For $u \in [0,s]$, it is then easy to observe from direct manipulations on \eqref{mild-t}, that the following flow property holds:
\begin{equation}
\label{eq-flow property}
    X_{s}^{0,x,I} = e^{(s-u)\A}X_{u}^{0,x,I} + \int_{u^-}^{s}e^{(s-r)\A}\sigma \d W_{r}+ \int_{u^-}^{s}e^{(s-r)\A}\d I_{r} = X_{s}^{u,X_{u}^{t,x,I},I}.
\end{equation}
\vspace{0.25cm}

\emph{Step 1.} We first provide a sketch of the proof of a weaker version of \eqref{eq:DP}, namely for any time $\eta \in [0,\infty)$ it holds
\begin{equation}
\label{eq:DP-weak}
V(x)=\inf_{I\in\hat{\mathcal{I}_0}}\E\left[\int_{0^-}^{\eta} e^{-\rho t} \left(G(X^{x,I}_{t})\d t+\d \nu_{t}\right)+ e^{-\rho \eta} V(X_{\eta}^{x,I})\right], \quad x \in H.
\end{equation}
\vspace{0.10cm}

\textsl{Proof of "$\geq$" in \eqref{eq:DP-weak}.} As in \cite[Lemma 5.1]{DeAM} (see \cite[Lemma 1.99]{FGS} for the infinite-dimensional setting with regular controls), now denote by $\overline{I}$ the $\mathbb{F}^{t,0}$-predictable process such that $\overline{I}\in \hat{\mathcal{I}_t}$ and $\overline{I}=I$ $\mathbb{P}\otimes \d t$-a.e.\ on $\Omega \times [t,\infty)$. Accordingly, we set $\overline{X}_{s}^{t,x,\overline{I}}$ to be given by \eqref{mild-t} with $\overline{I}$ instead of $I$. Then, using that $\overline{X}_{s}^{t,x,\overline{I}}$ and ${X}_{s}^{t,x,{I}}$ are indistinguishable and exploiting the flow property \eqref{eq-flow property} (now for the mild solution $\overline{X}$), we obtain 
\begin{eqnarray}
\label{eq:DDP-partI}
    && \hspace{2cm} \E \bigg[\int_{0^-}^{\infty}e^{-\rho t} \Big({G}(X_t^{0,x,I})\d t +\d \nu_{t} \Big)\bigg]  \\
    && = \E \bigg[\int_{0^-}^{\eta}e^{-\rho t} \Big({G}({X}_t^{0,x,I})\d t +\d {\nu}_{t} \Big)\bigg] + \E \bigg[\int_{\eta^-}^{\infty}e^{-\rho t} \Big({G}(\overline{X}_t^{0,x,\overline{I}})\d t +\d \overline{\nu}_{t} \Big)\bigg] \nonumber  \\
    && =  \E\bigg[\int_{0^-}^{\eta}e^{-\rho t} \Big({G}({X}_t^{0,x,I})\d t +\d {\nu}_{t} \Big)\bigg] + \E \bigg[e^{-\rho \eta}\E\bigg[\int_{\eta^-}^{\infty}e^{-\rho (t -\eta)} \Big({G}(\overline{X}_t^{0,x,\overline{I}})\d t +\d \overline{\nu}_{t} \Big)\Big | \, \mathcal{F}^{0,0}_{\eta}\bigg]\bigg] \nonumber \\
    && =   \E\bigg[\int_{0^-}^{\eta}e^{-\rho t} \Big({G}({X}_t^{0,x,I})\d t +\d {\nu}_{t} \Big)\bigg] + \E \bigg[e^{-\rho \eta}\E\bigg[\int_{\eta^-}^{\infty}e^{-\rho (t -\eta)} \Big({G}(X_{t}^{\eta,X_{\eta}^{0,x,\overline{I}},\overline{I}})\d t +\d \overline{\nu}_{t} \Big)\Big | \, \mathcal{F}^{0,0}_{\eta}\bigg]\bigg]. \nonumber 
\end{eqnarray}
    
    In order to take care of the conditional expectation in the last display equation, we proceed as in Equation (2.29) in \cite{FGS} (see also Step 1 in the proof of \cite[Thm.\ 4.2]{DeAM}). In particular, we define the regular conditional probability on $(\Omega,\mathcal{F})$ given by $\mathbb{P}_{\omega}(A):= \mathbb{P}(A\,|\,\mathcal{F}^{0,0}_{\eta})$ and for any $A \in \mathcal{F}$ $\mathbb{P}$-a.e.\ $\omega$, and we denote by $\mathbb{E}_{\omega}$ its expectation. Moreover, we introduce the control system
    $$I^{\omega}(\cdot):=(\Omega, \mathcal{F}_{\omega}, \mathbb{P}_{\omega}, \mathbb{F}^{\omega, \overline{\tau}_{\varepsilon}}:=(\mathcal{F}^{\omega, \eta}_s)_{s \in [\eta,\infty)}, (W_s - W_{\eta})_{s \in [\eta,\infty)}, \overline{I}|_{[\eta^-, \infty)}(\cdot)) \in \hat{\mathcal{I}}_{\eta},$$
    where $\mathcal{F}^{\omega}$ is the completion of $\mathcal{F}$ with $\mathbb{P}^{\omega}$-null sets and $\mathbb{F}^{\omega, \eta}$ the augmentation with $\mathbb{P}^{\omega}$-null sets of the sigma-algebra generated by the increments $W_{\cdot} - W_{\eta}$, 
    
    Then, appealing to standard properties of regular conditional probabilities (see the remarks at Page 102 in \cite{FGS}), law-invariance property of the state process (see also \cite[Hypothesis 2.29]{FGS}), and recalling the definition of the cost functional \eqref{eq:costfunctbis}, we can write
   \begin{eqnarray*}
     && \E\bigg[\int_{\eta^-}^{\infty}e^{-\rho (t -\eta)} \Big({G}(X_{t}^{\eta,X_{\eta}^{0,x,\overline{I}},\overline{I}})\d t +\d \overline{\nu}_{t} \Big)\Big | \, \mathcal{F}^{0,0}_{\eta}\bigg] \nonumber \\
     && = \E_{\omega}\bigg[\int_{\eta^-}^{\infty}e^{-\rho (t -\eta)} \Big({G}(X_{t}^{\eta,X_{\eta}^{0,x,\overline{I}},\overline{I}})\d t +\d \overline{\nu}_{t} \Big)\bigg] = \mathcal{J}(\overline{X}_{\eta}^{0,x,{I}}(\omega); I^{\omega}) \geq V(\overline{X}_{\eta}^{0,x,{I}}(\omega)).
   \end{eqnarray*}
   Hence, using this last display equation into \eqref{eq:DDP-partI} we obtain
  \begin{eqnarray}
\label{eq:DDP-partI-bis}
     && \E \bigg[\int_{0^-}^{\infty}e^{-\rho t} \Big({G}(X_t^{0,x,I})\d t +\d \nu_{t} \Big)\bigg] \nonumber  \\
    && \geq \E \bigg[\int_{0^-}^{\eta}e^{-\rho t} \Big({G}({X}_t^{0,x,I})\d t +\d {\nu}_{t} \Big) + e^{-\rho \eta} V(\overline{X}_{\eta}^{0,x,{I}})\bigg]
\end{eqnarray}
By taking the infimum over all $I\in\hat{\mathcal{I}_0}$ it thus follows that
\begin{equation}
\label{eq:DPP-largerequal}
   V(x) \geq \inf_{I:=(\vartheta,\nu)\in\hat{\mathcal{I}_0}}\E\left[\int_{0^-}^{\eta} e^{-\rho t} \left(G(X^{x,I}_{t})\d t+\d \nu_{t}\right)+ e^{-\rho \eta} V(X_{\eta}^{x,I})\right].
\end{equation}
\vspace{0.10cm}

\emph{Proof of $\leq$ in \eqref{eq:DP-weak}}. Given the continuity of $V$ (see Proposition \ref{prop:regularityV}) and the continuity of the control functional $\mathcal{J}(\cdot\,;I)$ (due to semiconcavity, see the proof of item (iv) in the proof of Proposition \ref{prop:regularityV}), one can argue as in Part 2 of the proof of \cite[Thm.\ 2.24]{FGS} (see also Step 2 in the proof of \cite[Thm.\ 4.2]{DeAM} for the finite-dimensional setting with singular controls) to conclude.
\vspace{0.25cm}

\emph{Step 2.} Having the weaker version \eqref{eq:DP-weak} and given the continuity of $V$ implied by its semiconcavity, one can then argue as in the proof of \cite[Th.\ 3.70]{FGS} (whose arguments still hold in the case of singular controls) to conclude the validity of \eqref{eq:DP}.


\section{Some technical results}
\label{sec:appendixB}

\renewcommand{\theequation}{B-\arabic{equation}}

\begin{lemma}
\label{lem:convdom}
Let $x \in H$, $I\in \mathcal{I}$, $(\mathcal{A}_{n})_{n\in \mathbb{N}}$ be the Yosida approximants of $\mathcal{A}$, and denote by $X^{n;x,I}$ the unique mild solution to 
$$
\d X^{n;x,I}_{t}=\mathcal{A}_{n} X^{n;x,I}_{t}\d t+\sigma \d W_{t}+\d I_{t}, \ \ \ X_{0^{-}}=x \in H;
$$
that is,
\begin{equation}
\label{eq:mildsolapp}
X_{t}^{n;x,I}=e^{t\A_{n}}x+W_{t}^{\A_{n},\sigma}+ \int_{0^-}^{t}e^{(t-s)\A_{n}}\d I_{s},\quad t\geq 0.
\end{equation}

For all $T>0$ there exists $M>0$ such that 
\begin{enumerate}[(i)]
\item[]
\item $\E\big[\sup_{t\in[0,T]}\big|X^{n;x,I}_t\big|_H^2\big] \leq M\Big(1 + |x|_H^2 + \E\Big[\Big|\int_0^T \big|\hat{\vartheta}_s\big|_H \d |I|_s \Big|^2\Big]\Big)$;
\vspace{0.2cm}

\item $\lim_{n\uparrow \infty}\E\big[\sup_{t\in[0,T]}\big|X^{n;x,I}_t - X^{x,I}_t\big|_H^2] =0.$
\end{enumerate}
\end{lemma}

\begin{proof}
We start by proving (i). Estimating the Brownian integral in \eqref{eq:mildsolapp} using Theorem 1.112 in \cite{FGS}, and an estimate analogous to \eqref{estimatesup-W} for the integral involving the singular control, we have, for some constant $M>0$ (possibly depending on $T$ and changing from line to line),

\begin{eqnarray}
\label{estimate1sup}
\E\big[\sup_{t\in[0,T]}\big|X^{n;x,I}_t\big|_H^2\big] & \leq & M\Big(|x|_H^2 + \E\Big[\sup_{t\in[0,T]}\Big|\int_0^t e^{(t-s)\mathcal{A}_n}\sigma \d W_s\Big|_H^2\Big]  \nonumber \\
&& + \E\Big[\sup_{t\in[0,T]}\Big|\int_0^t e^{(t-s)\mathcal{A}_n}\hat{\vartheta}_s \d |I|_s\Big|_H^2\Big] \Big) \nonumber \\
&& \leq  M\Big(|x|_H^2 + \E\Big[\int_0^T |\sigma\sigma^*|_{\mathcal{L}_1(H)} \d s\Big] +  \E\Big[\Big|\int_0^T \big|\hat{\vartheta}_s\big|_H \d |I|_s \Big|^2\Big] \Big)\nonumber \\
&& \leq M\Big(1 + |x|_H^2 + \E\Big[\Big|\int_0^T \big|\hat{\vartheta}_s\big|_H \d |I|_s \Big|^2\Big]\Big), \nonumber
\end{eqnarray}
where the last expectation is finite due to the fact that $I \in \mathcal{I}$ (cf.\ \eqref{set:S}).
This proves the first claim.

As for (ii) notice that, for a constant $M>0$ changing from line to line,
\begin{eqnarray*}
\label{estimate2sup}
\E\Big[\sup_{t \in [0,T]} \big|X^{n;x,I}_t - X^{x,I}_t\big|_H^2\Big] & \leq & M \Big(\sup_{t \in [0,T]}\big|\big(e^{t\mathcal{A}_n} - e^{t\mathcal{A}}\big)x\big|^2_H \nonumber \\
&& + \E\Big[\sup_{t \in [0,T]}\Big|\int_0^t \Big(e^{(t-s)\mathcal{A}_n} - e^{(t-s)\mathcal{A}}\Big)\sigma \d W_s\Big|_H^2\Big] \nonumber \\
&& + \E\Big[\sup_{t\in[0,T]}\Big|\int_0^t \Big(e^{(t-s)\mathcal{A}_n} - e^{(t-s)\mathcal{A}}\Big)\hat{\vartheta}_s \d |I|_s\Big|_H^2\Big]\Big). 
\end{eqnarray*}

The first two addends on the right-hand side of the latter display equation converge to zero as $n \uparrow \infty$ as in the proof of \cite[Thm.\ 1.112]{FGS}.
In order to deal with the third addend, define
$$\psi_n(s):=\sup_{t\in[s,T]}\Big|\Big(e^{(t-s)\mathcal{A}_n} - e^{(t-s)\mathcal{A}}\Big)\hat{\vartheta}_s\Big|_H, \quad s \in [t,T],$$
which is such that $\psi_n(s) \to 0$ as $n\uparrow \infty$ $\P$-a.s.\ by \cite[Prop.\ B.34]{FGS}.
Since now $|\psi_n(s)| \leq 2|\hat{\vartheta}_s|_H$, and, for any $T>0$, $\int_0^T |\hat{\vartheta}_s|_H \d|I|_s < \infty$ $\P$-a.s.\ by \eqref{DP} and $\E[|\int_0^T |\hat{\vartheta}_s|_H \d |I|_s|^2] < \infty$ because $I \in \mathcal{I}$, the dominated convergence theorem gives
\begin{eqnarray*}
&& \lim_{n\uparrow \infty}\E\Big[\sup_{t\in[0,T]}\Big|\int_0^t \Big(e^{(t-s)\mathcal{A}_n} - e^{(t-s)\mathcal{A}}\Big)\hat{\vartheta}_s \d |I|_s\Big|_H^2\Big] \leq \lim_{n\uparrow \infty} \E\Big[\sup_{t\in[0,T]}\Big|\int_0^t \psi_n(s) \d |I|_s\Big|_H^2\Big]  =0 \nonumber 
\end{eqnarray*}
\end{proof}

By arguing as in the proof of Lemma \ref{lem:convdom}(i) one can also prove the following.

\begin{lemma}
\label{lem:convdom2}
Let $x \in H$, $I \in \mathcal{I}$ and let $X^{x,I}$ denote the unique mild solution to \eqref{eq:state}; that is,
$$X_{t}^{x,I}=e^{t\A}x+W_{t}^{\A,\sigma}+ \int_{0^-}^{t}e^{(t-s)\A} \d I_{s},\quad t\geq 0.$$
For any $T>0$, there exists  $M>0$ such that
$$\E\big[\sup_{t\in[0,T]}\big|X^{x,I}_t\big|_H^2\big] \leq M\Big(1 + |x|_H^2 + \E\Big[\Big|\int_0^T \big|\hat{\vartheta}_s\big|_H \d |I|_s \Big|^2\Big]\Big).$$
\end{lemma}


\end{document}